\documentclass[10pt]{amsart}
\usepackage{graphicx}
\usepackage{latexsym}
\usepackage{fancyhdr}
\usepackage{amsmath, amssymb}
 \usepackage[utf8]{inputenc}
\usepackage[all]{xy}
\usepackage{pdflscape}
\usepackage{longtable}
\usepackage{placeins}
\usepackage{float}
\usepackage{rotating}
\usepackage{tikz}
\usepackage{hyperref}
\usepackage{subfigure}
\usepackage{mathrsfs}

\usepackage{tensor}
\usepackage{enumerate}
\usepackage{mathtools}

\usepackage{color}
\usetikzlibrary{patterns}

\newlength{\hatchspread}
\newlength{\hatchthickness}
\newlength{\hatchshift}
\newcommand{\hatchcolor}{}
\tikzset{hatchspread/.code={\setlength{\hatchspread}{#1}},
         hatchthickness/.code={\setlength{\hatchthickness}{#1}},
         hatchshift/.code={\setlength{\hatchshift}{#1}},
         hatchcolor/.code={\renewcommand{\hatchcolor}{#1}}}
\tikzset{hatchspread=2pt,
         hatchthickness=0.2pt,
         hatchshift=0pt,
         hatchcolor=black}
\pgfdeclarepatternformonly[\hatchspread,\hatchthickness,\hatchshift,\hatchcolor]
   {custom north east lines}
   {\pgfqpoint{\dimexpr 0\hatchthickness}{\dimexpr0\hatchthickness}}
   {\pgfqpoint{\dimexpr\hatchspread+2\hatchthickness}{\dimexpr\hatchspread+2\hatchthickness}}
   {\pgfqpoint{\dimexpr\hatchspread}{\dimexpr\hatchspread}}
   {
    \pgfsetlinewidth{\hatchthickness}
    \pgfpathmoveto{\pgfqpoint{\dimexpr\hatchshift-0.15pt}{-0.15pt}}
    \pgfpathlineto{\pgfqpoint{\dimexpr\hatchspread+0.15pt}{\dimexpr\hatchspread-\hatchshift+0.15pt}}
    \ifdim \hatchshift > 0pt
      \pgfpathmoveto{\pgfqpoint{-0.15pt}{\dimexpr\hatchspread-\hatchshift-0.15pt}}
      \pgfpathlineto{\pgfqpoint{\dimexpr\hatchshift+0.15pt}{\dimexpr\hatchspread+0.15pt}}
    \fi
    \pgfsetstrokecolor{\hatchcolor}
    \pgfusepath{stroke}
   }

\theoremstyle{plain}
\newtheorem{thm}{Theorem}[section]

\newtheorem{cor}[thm]{Corollary}
\newtheorem{lem}[thm]{Lemma} 
\newtheorem{prop}[thm]{Proposition}

\theoremstyle{definition}
\newtheorem{defi}[thm]{Definition}
\theoremstyle{remark}
\newtheorem{rem}[thm]{Remark}
\numberwithin{equation}{section}

\newtheorem{ex}[thm]{Example}

\definecolor{aquamarine}{rgb}{0.5, 1.0, 0.83}
\definecolor{bananamania}{rgb}{0.98, 0.91, 0.71}
\definecolor{blizzardblue}{rgb}{0.67, 0.9, 0.93}
\definecolor{corn}{rgb}{0.98, 0.93, 0.36}
\definecolor{lightgreen}{rgb}{0.76, 0.98, 0.76}
\definecolor{lightcoral}{rgb}{0.97, 0.75, 0.75}
\definecolor{lightblue}{rgb}{0.68, 0.85, 0.9}

\newcommand{\haut}{\operatorname{ht}}
\newcommand{\lgw}{\longrightarrow}
\newcommand{\lgm}{\longmapsto}

\newcommand{\ovl}{\overline}

\newcommand{\F}{\mathbb F}
\newcommand{\Ra}{\operatorname{Raf}^-}
\newcommand{\Raf}{\operatorname{Raf}^+}
\newcommand{\T}{\operatorname{\mathsf T}}

\newcommand{\Ker}{\operatorname{Ker}}

\newcommand{\Inn}{\operatorname{Inn}}
\newcommand{\Out}{\operatorname{Out}}

\newcommand{\Gal}{\operatorname{Gal}}

\newcommand{\Vect}{\operatorname{Vect}}

\newcommand{\G}{\Gamma}
\newcommand{\la}{\lambda}
\renewcommand{\O}{\mathcal{O}}

\renewcommand{\L}{\mathbb{L}}
\newcommand{\Aut}{\operatorname{Aut}}

\newcommand{\Z}{\mathbb{Z}}

\newcommand{\rr}{\operatorname{rat.rank}}
\newcommand{\tdr}{\operatorname{tr.deg}_\k}
\renewcommand{\k}{\Bbbk}

\newcommand{\Mat}{\operatorname{Mat}}
\newcommand{\R}{\mathbb{R}}
\newcommand{\U}{\mathcal U}
\newcommand{\V}{\mathcal V}
\newcommand{\LEQ}{\,\boldsymbol{\leq}\,}

\newcommand{\m}{\mathfrak m}
\newcommand{\SEQ}{\,\boldsymbol{<}\,}
\newcommand{\K}{\mathbb{K}}

\newcommand{\Tor}{\operatorname{Tor}}
\renewcommand{\SS}{\mathbb{S}}

\newcommand{\N}{\mathbb{N}}

\newcommand{\he}{\operatorname{ht}}
\newcommand{\Q}{\mathbb{Q}}

\newcommand{\p}{\mathfrak{p}}

\renewcommand{\lg}{\langle}
\newcommand{\s}{\sigma}
\newcommand{\rg}{\rangle}
\newcommand{\q}{\mathfrak{q}}
\renewcommand{\a}{\alpha}

\newcommand{\rk}{\operatorname{rank}}

\renewcommand{\phi}{\varphi}
\renewcommand{\d}{\delta}
\newcommand{\ZR}{\operatorname{ZR}}
\newcommand{\Ord}{\operatorname{Ord}}

\newcommand{\e}{\varepsilon}

\begin{document}
\baselineskip=13pt
%
\title{Preordered groups  and valued fields}

\author{Julie Decaup}
\email{julie.decaup@im.unam.mx}
\address{Instituto de Matem\'aticas, Universidad Nacional Aut\'onoma de M\'exico (UNAM), Mexico}

\author{Guillaume Rond}
\email{guillaume.rond@univ-amu.fr}
\address{Aix-Marseille Universit\'e, CNRS, Centrale Marseille, I2M, UMR 7373, 13453 Marseille, France}


\thanks{
The second author is deeply grateful to the UMI LASOL of  CNRS where this project has been carried out.}

\begin{abstract}
We study  algebraic, combinatorial and topological properties of the set of preorders on a group, and the set of valuations on a field. We show strong analogies between these two kinds of sets and develop a dictionary for these ones. Among the results  we make a detailed study of the set of preorders on $\Z^n$. We also prove that the set of valuations on a countable field of transcendence degree at least 2 is an ultrametric Cantor set.
\end{abstract}
\maketitle
\tableofcontents
\section{Introduction}
The purpose of this paper is to investigate some algebraic, combinatorial and topological properties of spaces of preorders on a given group, and spaces of valuations on a given field. In particular we show that these spaces shares very strong similarities, and we develop a dictionary between preorders on groups and valuations on fields.\\
\\
Historically, the study of orderable groups has been developed since the  end of the nineteen century for their importance in algebraic topology. But the first study of the topological properties of the set of orders on a group is due to Kuroda in the case $G=\Z^n$ \cite{K}, and to Sikora in the general case \cite{S}. Here, an order means a total order that is left-invariant. In his paper, Sikora introduced a topology on the set of orders on a group, and showed that this topology is a metric topology in the case of countable groups. For a countable group $G$, Sikora proved that the space of left-invariant orders (denoted by $\Ord_l(G)$) on $G$ is a compact metric space, and shows that this is even a Cantor set when $G=\Z^n$. Subsequently, several authors proved that $\Ord_l(G)$ is a Cantor set for several examples of groups $G$.  \\
The first study (to our knowledge) of the space of preorders on a group $G$ is due to Ewald and Ishida \cite{EI} for $G=\Z^n$. Let us mention that a preorder satisfies all the properties of an order except that it may not be antisymmetric. In their paper, they introduce a topology of the set of preorders on $\Z^n$ (extending the one of Kuroda), and show the compacity of this set. \\
\\
On the other hand, Zariski introduced a topology on the set of valuations of a field (called the Zariski-Riemann space), proved its compacity and used this in order to deduce the  resolution of singularities in dimension two from the local uniformization theorem (see \cite{Za1} and \cite{Za}). The study of valuation theory has been revived in the last twenty years for its applications in commutative algebra and algebraic geometry (see \cite{Va} or \cite{HS} for example).\\
\\
In this paper we begin by studying preorders on groups. First, we show that the set of left-invariant preorders on a group $G$ (denoted by $\ZR_l(G)$) is equipped with a natural order that makes $\ZR_l(G)$ a join-semilattice (see Theorem \ref{cor_raf_toset}) and even a rooted graph under some assumptions on $G$ (see Proposition \ref{tree} and Corollary \ref{tree2}). 
Then we introduce and investigate three topologies on $\ZR_l(G)$: the Zariski topology, the Inverse topology and the Patch topology. These correspond to the topologies having the same name on the set of valuations on a given field and introduced by Zariski (see \cite{SZ} or \cite{H}). Moreover the Patch topology coincides with the Chabauty topology on the sets of submonoids of $G$, where $G$ is endowed with the cofinite topology. We prove that $\ZR_l(G)$ is compact for these three topologies (using the same argument as Zariski for the case of spaces of valuations), see Theorem \ref{compacity_preorders}. The first two topologies are not metric, but we show, following Sikora, that  the last one is ultrametric when $G$ is countable (see Proposition \ref{prop_metric}). Let us mention that these three topologies coincide on the subset of orders and, therefore, correspond to the topology introduced by Sikora.  Then we study in more details the case of abelian groups, and we make a detailed study of $\ZR(\Z^n)$: we show that this is not a Cantor set in general, but that it contains infinitely many explicit Cantor subsets when $n\geq 2$, generalizing the result of Sikora (see Theorem \ref{main1}). In fact the set $\ZR(\Z^n)$ can be seen as a rooted graph on which acts $\Aut(\Q^n)$. More precisely $\ZR(\Z^n)$ can be seen as follows: we consider the rooted graph $T_0$ that has one root $\leq_\emptyset$ and a set of vertices in bijection with the sphere of dimension $n-1$, and the edges are the pairs $(\leq_\emptyset,\preceq)$ where $\preceq$ runs over the other vertices. Then $\ZR(\Z^n)$
is obtained by gluing $T_0$ with infinitely many copies of the $\ZR(\Z^d)$ for $d<n$. From this we deduce an effective version of Hausdorff-Alexandroff Theorem for the spheres of any dimension (Proposition \ref{HA}). The case of non-commutative groups is much more difficult in general. We provide two examples: the description of $\ZR(G)$ when $G$ is the fundamental group of the Klein bottle, and we give an example of a torsion free group $G$ for which $\ZR(G)$ is trivial.\\
Then we develop the analogy with the set of valuations on a field $\K$. We denote by $\ZR(\K/\k)$ the set of valuations on $\K$ that are trivial on the subfield $\k$. Here again, $\ZR(\K/\k)$ is a join-semilattice (Proposition \ref{jsl_val}).  When $\K$ is a countable field, we show in an explicit way that the Patch topology is an ultrametric topology (see Theorem \ref{patch_metric}). The main difference with $\ZR_l(G)$ for a group $G$, is that the subfield $\k$ plays the role of the trivial subgroup $\{1\}$, but $\k$ is not finite in general. Therefore several difficulties appear. For example, the Zariski, Inverse and Patch topologies do not coincide in general on the set of rational valuations, but they do when $\k$ is a finite field (see Proposition \ref{refi_top2}). Then, by analogy with the case of orders and preorders on $\Z^n$, we investigate when $\ZR(\K/\k)$ (or some subsets of it) are Cantor sets. First we prove an analogue of the result of Sikora: the set of rational valuations on $\k(x_1,\ldots, x_n)$ vanishing on $\k$ (when $\k$ is a finite field), is a Cantor set for the Zariski topology (see Theorem \ref{cantorset1}). When $\k$ is not finite, this set is unfortunately not closed, therefore not compact. But we prove  that $\ZR(\K/\k)$ is a Cantor set for the Patch topology when $\k\lgw \K$ is a finitely generated field extension of transcendence degree at least 2 and $\k$ is at most countable (Theorem \ref{cantorset2}).\\
The dictionary between preorders on a group $G$ and valuations on a field $\K$ can be summarized in the following table (the corresponding  objects will be introduced all along the paper):


\begin{table}[h]
   \centering
\begin{tabular}{|c||c|}
  \hline

Group $G$ & Field $\K$\\
  \hline
monoid $S$ & Ring $R$\\  
  \hline
  Preorder $\preceq$ & Valuation $\nu$ \\
  \hline
 Preorder monoid $V_\preceq$ & $\quad$ Valuation ring $V_\nu\quad$\\
  \hline
 maximal ideal $\m_\preceq$ & maximal ideal $\m_\nu$\\
  \hline
  $\rk(\preceq)$ & $\rk(\nu)$\\
  \hline
  Residue group  $G_\preceq$ & Residue field $\k_\nu$\\
  \hline
  $\deg(\preceq)$ & $\tdr(\k_\nu)$ \\
  \hline
  $\O_u$  & $\O(x)$\\
  \hline
  $\U_u$ & $\U(x)$\\
  \hline

\end{tabular}
\end{table}

Let us mention that this analogy  has been emphasized in the case  of preorders on $\Z^n$  and valuations on $\k(x_1,\ldots, x_n)$ in \cite{EI}, where the authors extend the Zariski topology to the set of preorders on $\Z^n$ and show its compacity, and in \cite{T} where the author provides a new proof of the fact that the set of orders on $\Z^n$ is a Cantor set. This work has been motivated by our previous work where we used in an essential way the compacity of $\ZR(\Z^n)$ \cite{ADR}.

\section{The Zariski-Riemann space of preorders}
\subsection{Generalities}
\begin{defi}\label{preorder}
Let $G$ be a group. We denote by $\ZR_l(G)$ the set of left-invariant preorders on $G$, i.e. the set of binary relations $\preceq$ on $G$ such that
\begin{itemize}
\item[i)] $\forall u, v\in G$, $u\preceq v$ or $v\preceq u$,
\item[ii)] $\forall u,v,w\in G$, $u\preceq v$ and $v\preceq w$ implies $u\preceq w$,
\item[iii)] (\emph{left invariance}) $\forall u,v,w\in G$, $u\preceq v$ implies $wu\preceq wv$.
\end{itemize}
In the same way, we define right-invariant preorders whose set is denoted by $\ZR_r(G)$. The set of preorders that are bi-invariant, that is left and right-invariant, is denoted by $\ZR(G)$.
The subset of orders of $\ZR_*(G)$  is denoted by $\Ord_*(G)$ for $*=l$, $r$ or $ \emptyset$.\\
The trivial preorder, i.e. the unique preorder $\preceq$ such that $u\preceq v$ for every $u$, $v\in G$ is denoted by $\leq_\emptyset$.
\end{defi}

\begin{rem}
If $G$ is an abelian group then $\ZR(G)=\ZR_l(G)=\ZR_r(G)$.
\end{rem}

\begin{rem}
There is a bijection between $\ZR_l(G)$ and $\ZR_r(G)$ defined as follows:\\
For $\preceq\in\ZR_l(G)$ we define the right-invariant preorder $\preceq'$ by
$$\forall u,v\in G,\ u\preceq' v\Longleftrightarrow v^{-1}\preceq u^{-1}.$$
So from now on, we will no longer consider right-invariant preorders.
\end{rem}

\begin{defi}
Let $G$ be a group.
For $\preceq\in\ZR_*(G)$ and $u$, $v\in G$, we write $u\prec v$ if 
$$u\preceq v\text{ and }\neg(v\preceq u).$$
\end{defi}

\begin{defi}\label{def_equiv}
Let $G$ be a group.
Let $\preceq\in \ZR_*(G)$. We define a congruence relation $\sim_\preceq$ as follows: 
$$u\sim_\preceq v \text{ if } u\preceq v \text{ and } v\preceq u.$$
This congruence relation is compatible with the group law if $\preceq$ is bi-invariant. In this case the quotient $G/\sim_\preceq$ is a group denoted by $G^\preceq$ and $\preceq$ induces in an obvious way an order on $G^\preceq$ still denoted by $\preceq$.

\end{defi}

\begin{lem}\label{lem1}
Let $G$ be a group.
Let $a$, $b$, $c$, $d\in G$ with $a\preceq b$ and $c\preceq d$.
\begin{enumerate}
\item If $\preceq\in\ZR(G)$ then $ac\preceq bd$.
\item If $\preceq\in\ZR_l(S)$, then $b^{-1}ad^{-1}c\preceq 1.$
\end{enumerate}
If we assume moreover that $a\prec b$ or $c\prec d$, then we obtain strict inequalities in both cases.
\end{lem}

\begin{proof}
For the first inequality just remark that $ac\preceq bc\preceq bd$.\\
Now if  $\preceq\in\ZR_l(G)$ we have $b^{-1}a\preceq 1$ and $d^{-1}c\preceq 1$.
Therefore we have $b^{-1}ad^{-1}c\preceq b^{-1}a\preceq 1$.
\\
Let us remark that if $a\prec b$, $\preceq\in \ZR_l(G)$, then $ca\prec cb$ for every $c\in G$. Thus the cases of strict inequalities are proved as the previous cases.
\end{proof}

\begin{rem}
Let $G$ be a group.
If $\Tor(G)\neq\{1\}$ then $\Ord_*(G)=\emptyset$. Indeed let $u\in\Tor(G)$  be of order $n\geq 1$ and let $\preceq\in\Ord_*(G)$. Assume $1\preceq u$. Thus
$$1\preceq u\preceq u^2\preceq \cdots\preceq u^n=1.$$
Since $\preceq$ is an order, we have $u=1$.  The same is true if we assume $u\preceq 1$. Therefore $\Tor(G)=\{1\}$.
\end{rem}

\subsection{Preorder monoid}
Let $G$ be a group and $\preceq\in\ZR_l(G)$. We set
$$V_\preceq=\{u\in G\mid u\succeq 1\},$$
$$\m_\preceq:=\{u\in G\mid u\succ 1\}.$$
It is straightforward to check that $V_\preceq$ is a monoid, and $\m_\preceq$ an ideal of $V_\preceq$ (by Lemma \ref{lem1}). Moreover $V_\preceq$ is a preorder monoid:
\begin{defi}\label{semigroup_valuation}
Let $G$ be a group and $V$ be a sub monoid of $G$. We say that $V$ is a \emph{preorder monoid} if
\begin{equation}\label{semigp_preorder}\forall u\in G,\ \ u\in V \text{ or } u^{-1}\in V.\end{equation}
\end{defi}

On the other hand, if $V$ is a preorder monoid, then $V=V_\preceq$ where $\preceq\in \ZR_l(G)$ is defined as follows: for every $u$, $v\in G$, we set $u\succeq v$ if and only if $v^{-1}u\in V$. Since $V$ satisfies \eqref{semigp_preorder}, Definition \ref{preorder} i) is satisfied. Since $V$ is a monoid, Definition \ref{preorder} ii) is satisfied, and Definition \ref{preorder} iii) is automatically satisfied.\\
\\
Moreover $\preceq$ is bi-invariant if and only if $V_\preceq$ is a normal sub monoid of $G$, that is
$$\forall u\in V_\preceq, v\in G,\ \ v^{-1}uv\in V_\preceq.$$
\\
We remark that $\m_\preceq$ is the unique maximal ideal of $V_\preceq$ since  the inverse of every element of $ V_{\preceq}\setminus \m_\preceq$ is  in $V_{\preceq}$.

\begin{defi}
Let $G$ be a group and $\preceq\in \ZR_l(G)$. The monoid $V_\preceq$ is called the \emph{preorder monoid associated to $\preceq$}, and $\m_\preceq$ is its \emph{maximal ideal}.
\end{defi}


\subsection{Ordering of the set of orders}
\begin{defi}
Given two preorders $\preceq_1$, $\preceq_2\in\ZR_l(G)$ where $G$ is a group, we say that $\preceq_2$ refines $\preceq_1$ if
$$\forall u,v\in G, u\preceq_2 v\Longrightarrow u\preceq_1 v.$$
\end{defi}

\begin{rem}\label{rem1}
Let $\preceq_1$, $\preceq_2\in\ZR_l(G)$. If $\preceq_1$ refines $\preceq_2$ and $\preceq_2$ refines $\preceq_1$ then $\preceq_1=\preceq_2$.
\end{rem}

\begin{rem}\label{contr}
By contraposition, $\preceq_2$ refines $\preceq_1$ if and only if 
$$\forall u,v\in G, u\prec_1 v\Longrightarrow u\prec_2 v.$$
\end{rem}


\begin{defi} Let $G$ be a group.
We define an order  $\boldsymbol{\leq}$ on $\ZR_l(G)$ as follows: for every preorders $\preceq_1$, $\preceq_2\in\ZR_l(G)$ we have
$$\preceq_1\ \boldsymbol{\leq} \ \preceq_2$$
if $\preceq_2$ is a refinement of $\preceq_1$. By Remark \ref{rem1} it is straightforward to check that $\boldsymbol{\leq}$ is an order.
\end{defi}

\begin{lem}
Let $G$ be a group. Given $\preceq$, $\preceq'\in\ZR_l(G)$, the following properties are equivalent:
\begin{itemize}
\item[i)] $\preceq\LEQ \preceq'$
\item[ii)] $V_{\preceq'}\subset V_\preceq$
\item[iii)] $\m_{\preceq}\subset \m_{\preceq'}$
\end{itemize}
\end{lem}

\begin{proof}
Assume i) holds, and let $u\in V_{\preceq'}$, that is $u\succeq' 1$. Then $u\succeq 1$ and $u\in V_\preceq$. Thus ii) holds.\\
Now assume that ii) holds, and let $u\in\m_\preceq$, that is $u\succ 1$. Therefore $u^{-1}\prec 1$, that is $u^{-1}\notin V_{\preceq}$. Thus $u^{-1}\notin V_{\preceq'}$ and $u\in \m_{\preceq'}$. Thus iii) holds. \\
Finally, assume iii) holds. Let $u$, $v\in G$ with $u\prec v$, that is $u^{-1}v\succ 1$. Thus, $u^{-1}v\in \m_\preceq\subset\m_{\preceq'}$. Therefore $u\prec' v$, and $\preceq\LEQ\preceq'$.
\end{proof}
\begin{lem}
Let $G$ be a group and let $E\subset \ZR_l(G)$ be non empty. The set
$$A_E:=\left\{S\text{ sub monoid of } G\mid \bigcup_{\preceq\in E} V_\preceq\subset S\right\}$$
is non empty and contains a minimal element. This minimal element is a preorder monoid, and its associated preorder is denoted by $\preceq_{\inf E}$.\\
Moreover, if $E\subset \ZR(G)$, then $\preceq_{\inf E}$ is bi-invariant.
\end{lem}

\begin{proof}
We have that $G\in A_E$, therefore this set is not empty. We set $V:=\bigcap\limits_{ S\in A_E}S$. We have $\bigcup\limits_{\preceq\in E} V_\preceq\subset V$, and for every $u\in G$, $u\in V$ or $u^{-1}\in V$ since the $V_\preceq$ are preorder monoids. Moreover $V$ is a monoid since the $S$ are monoids. This proves the existence of $\preceq_{\inf E}$.\\
Now, if all the $V_\preceq$ are normal monoids, then $V$ is a normal monoid, and $\preceq_{\inf E}$ is bi-invariant.
\end{proof}

Let $G$ be a group and let $\preceq \in \ZR_*(G)$. We  set 
$$\Ra_*(\preceq):=\{\preceq'\in\ZR_*(G) \text{ such that } \preceq' \LEQ \preceq\},$$
$$\Raf_*(\preceq):=\{\preceq'\in\ZR_*(G) \text{ such that } \preceq \LEQ \preceq'\}.$$

\begin{lem} \label{min et intersection}
Let $G$ be a group and $E\subset \ZR_*(G)$. We have
$$\bigcap_{\preceq\in E}\Ra_*(\preceq)  = \Ra_*\left(\preceq_{\inf E}\right).$$
\end{lem}

\begin{proof}
Indeed, we have
$$\preceq'\LEQ \preceq_{\inf E} \Longleftrightarrow V_{\preceq_{\inf E}}\subset V_\preceq'\Longleftrightarrow [\forall \preceq\in E,\ V_{\preceq}\subset V_{\preceq'}]\Longleftrightarrow [\forall \preceq\in E,\ \preceq'\LEQ\preceq].
$$
\end{proof}

\begin{rem}
If $E:=\{\preceq_1,\preceq_2\}$, then $\preceq_{\inf E}$ is denoted by $\preceq_1 \wedge \preceq_2$.
\end{rem}

\begin{lem}\label{lem2}
Let $G$ be a group. 
Let $\preceq_1$, $\preceq_2\in\ZR_l(G)$, none of them refining the other one. Then there is $u\in G$ such that
$$u\prec_1 1\text{ and } 1\prec_2 u.$$
\end{lem}

\begin{proof}
Because $\preceq_2$ is not refining $\preceq_1$ there are $a$, $b\in G$ such that
$a\preceq_2 b$  and $b\prec_1 a$.
By symmetry there are $c$, $d\in G$ such that
$c\preceq_1 d$ and $ d\prec_2 c$. 
Set $u=a^{-1}bd^{-1}c$. Then $u\prec_1 1$  by Lemma \ref{lem1}.  By symmetry we have $1\prec_2 u$.
\end{proof}

Therefore we have:

\begin{thm}\label{cor_raf_toset}
Let $G$ be a group. Then $\ZR_*(G)$ is a join-semilattice, that is a partially ordered set in which all subsets have  an infimum.\\
Moreover, for every  $\preceq\in\ZR_*(G)$,
 $(\Ra_*(\preceq), \,\boldsymbol{\leq}) $ is a totally ordered set.
\end{thm}

\begin{proof}
We have that $\ZR_*(G)$ is a join-semilattice by Lemma \ref{min et intersection}.\\
Now let $\preceq \in \ZR_*(G)$.
Let $\preceq_1$, $\preceq_2\in\Ra_*(\preceq)$, $\preceq_1\neq \preceq_2$. Assume, aiming for contradiction, that $\preceq_1\notin\Ra_*(\preceq_2)$ and $\preceq_2\notin \Ra_*(\preceq_1)$. Then by Lemma \ref{lem2} there exists $u\in G$ such that 
$u\prec_1 1$ and $u\succ_2 1.$
Since $\preceq$ refines $\preceq_1$ and $\preceq_2$ then
$u\prec 1$ and $u\succ 1$
which is a contradiction.
\end{proof}

\begin{prop}\label{tree}
Assume the following: 
\begin{itemize}
\item[i)] for every $\preceq\in\ZR_*(G)$, there is a maximal element $\preceq' \SEQ \preceq$.
\item[ii)] for every $\preceq_1$, $\preceq_2\in \ZR_*(G)$ with $\preceq_1\SEQ \preceq_2$, there is a minimal $\preceq'\in\ZR_*(G)$ such that
$$\preceq_1\SEQ\preceq'\LEQ \preceq_2.$$
\end{itemize}
In this case, $\ZR_*(G)$ has a rooted tree structure: the vertices of $\ZR_*(G)$ are the elements of $\ZR_*(G)$, and there is an edge between two vertices $\preceq_1$, $\preceq_2$ if $\preceq_1\SEQ\preceq_2$ or $\preceq_2\SEQ\preceq_1$ and there is no preorder between $\preceq_1$ and $\preceq_2$. The root is the trivial preorder.
\end{prop}

\begin{proof}
We only have to prove that, for every $\preceq_1$, $\preceq_2\in\ZR_*(G)$, there is a unique path connecting $\preceq_1$ to $\preceq_2$. By replacing $\preceq_2$ by $\preceq_1\wedge\preceq_2$, we may assume that $\preceq_2\LEQ \preceq_1$. Since $\Ra_*(\preceq_1)$ is totally ordered, we only need to prove that there is a path between $\preceq_2$ and $\preceq_1$. We consider the set
$$E:=\{\preceq\in\Ra_*(\preceq_1)\cap\Raf_*(\preceq_2)\mid \preceq\text{ is connected to }\preceq_1\}.$$
We claim that $\preceq_{\inf E}\in E$. Indeed, by ii), if $\preceq_{\inf E}\neq \preceq_1$, there is a minimal $\preceq\in \ZR_*(G)$ such that
$\preceq_{\inf E}\SEQ\preceq\LEQ\preceq_1.$
Therefore $\preceq\in E$, and $\preceq_{\inf E}$ is connected by an edge to $\preceq$.\\
Now, if $\preceq_2\neq \preceq_{\inf E}$, by i) there is $\preceq\in\ZR_*(G)$ such that
$\preceq_2\LEQ \preceq\SEQ\preceq_{\inf E},$
and $\preceq$ and $\preceq_{\inf E}$ are connected by an edge. This contradicts the definition of $E$. Therefore $\preceq_{\inf E}=\preceq_2$. This proves the result.
\end{proof}

\subsection{Topologies}

\subsubsection{The Zariski topology}

\begin{defi}
Let $G$ be a group. The Zariski  topology on $\ZR_*(G)$ (or Z-topology for short) is the topology
 for  which the sets
$$\mathcal O_{u}:=\{\preceq\in\ZR_*(G)\mid u\succeq 1\},$$
where $u$   runs over the elements of $G$, form a basis of open sets.  
\end{defi}

\begin{prop}\label{spe_ord}
Let $G$ be a group. The order  $\boldsymbol{\leq}$ is the specialization order of the topological set $\ZR_*(G)$, that is
$$\forall \preceq_1, \preceq_2\in \ZR_*(G),\  \preceq_1\, \boldsymbol{\leq}\,\preceq_2\  \Longleftrightarrow \ \ovl{\{\preceq_2\}}^{Z}\subset\ovl{\{\preceq_1\}}^{Z}$$
where $\ovl{E}^Z$ is the closure of $E\subset \ZR_*(G)$ for the Zariski topology.
\end{prop}

\begin{proof}
Let $\preceq_1\,\boldsymbol{\leq}\,\preceq_2$. If $\preceq_2\in\mathcal O_{u}$ then $u\succeq_2 1$ and $u\succeq_1 1$ since $\preceq_2$ refines $\preceq_1$. Thus $\preceq_1\in\mathcal O_{u}$. Hence every open set containing $\preceq_2$ contains $\preceq_1$. Hence $\preceq_2$ belongs to the Z-closure of $\{\preceq_1\}$.\\
On the other hand assume $\ovl{\{\preceq_2\}}^{Z}\subset\ovl{\{\preceq_1\}}^{Z}$. Let $u$, $v\in G$ such that $u\succeq_2 v$, that is $\preceq_2\in \mathcal O_{v^{-1}u}$. Since $ \mathcal O_{v^{-1}u}$ is open we have $\preceq_1\in \mathcal O_{v^{-1}u}$. Therefore $u\succeq_1v$ and $\preceq_1\, \boldsymbol{\leq}\,\preceq_2$.
\end{proof}

In particular this implies that for a given preorder $\preceq\in \ZR_*(G)$ we have
$\Raf_*(\preceq)=\ovl{\{\preceq\}}^Z$.

\subsubsection{The Inverse topology}

\begin{defi}
Let $G$ be a group.
The set $\ZR_*(G)$ is endowed with a topology for  which the sets
$$\mathcal U_{u}:=\{\preceq\in\ZR_*(G)\mid u\succ 1\},$$
where $u$  runs over the elements of $G$, form a basis of open sets. This topology is called the Inverse topology or I-topology. 
\end{defi}

\begin{rem}
For every $u\in G$,  $\mathcal O_{u}$ is the complement of  $\mathcal U_{u^{-1}}$.
\end{rem}

\begin{prop}\label{spe_ord2}
Let $G$ be a group. The order  $\boldsymbol{\leq}$ is the specialization inverse order of the topological set $\ZR_*(G)$, that is
$$\forall \preceq_1, \preceq_2\in \ZR_*(G),\  \preceq_1\, \boldsymbol{\leq}\,\preceq_2\  \Longleftrightarrow \ \ovl{\{\preceq_1\}}^{I}\subset\ovl{\{\preceq_2\}}^{I}.$$
\end{prop}

\begin{proof}
Let $\preceq_1\,\boldsymbol{\leq}\,\preceq_2$. If $\preceq_1\in\mathcal U_{u}$ then $u\succ_1 1$. Therefore $u\succ_2 1$ since $\preceq_2$ refines $\preceq_1$. Thus $\preceq_2\in\mathcal U_{u}$. Hence every open set containing $\preceq_1$ contains $\preceq_2$. Hence $\preceq_1$ belongs to the I-closure of $\{\preceq_2\}$.\\
On the other hand assume $\ovl{\{\preceq_1\}}^{I}\subset\ovl{\{\preceq_2\}}^{I}$. Let $u$, $v\in G$ such that $u\succ_1 v$, that is $\preceq_1\in \mathcal U_{v^{-1}u}$. Since $ \mathcal U_{v^{-1}u}$ is open we have $\preceq_2\in \mathcal U_{v^{-1}u}$. Therefore $u\succ_2v$ and $\preceq_1\, \boldsymbol{\leq}\,\preceq_2$.
\end{proof}

Therefore for a given preorder $\preceq\in \ZR_*(G)$ we have 
$\Ra_*(\preceq)=\ovl{\{\preceq\}}^I$.

%

\subsubsection{The Patch topology}
\begin{defi}
The Patch topology on $\ZR_*(G)$ (or P-topology for short) is the topology for which the sets $\U_{u}$ and $\O_{u}$, where $u$  runs over $G$, form a basis of open sets. This is the coarsest topology finer than the Zariski and the Inverse topologies.
\end{defi}



\subsubsection{Remarks about these topologies}
From now on, for a set $E\subset \ZR_*(G)$ where $G$ is a group, we say that that $E$ is $\star$-open if $E$ is open in the $\star$-topology for $\star=$ Z, I, or P. In the same way we define $\star$-continuous maps and $\star$-homeomorphisms.

\begin{prop}
Let $G$ be a group.
The  space $\ZR_*(G)$ endowed with the Zariski or the Inverse topology is $\T_0$, but it is not $\T_1$ when $\ZR_*(G)\neq \{\leq_\emptyset\}$. In particular $\ZR_*(G)$ is not metrizable for these two topologies.
\end{prop}

\begin{proof}
Let us prove the statement for the Zariski topology. \\
Let $\preceq_1\boldsymbol{\leq}\preceq_2$. Then $\preceq_2$ belongs to the closure of $\{\preceq_1\}$ by Proposition \ref{spe_ord} and $\ZR_*(G)$ is not $\T_1$.\\
Now let $\preceq_1$ and $\preceq_2$ two distinct preorders on $G$. In particular one of them does not refine the other. Assume for instance that $\preceq_2$ does not refine $\preceq_1$. Thus there exist $u$, $v\in\Z^n$ such that $u\preceq_2 v$ and $v\prec_1 u$. Thus $\preceq_2\in\mathcal O_{u^{-1}v}$ but $\preceq_1\notin\mathcal O_{u^{-1}v}$. Hence $\ZR_*(G)$ is $\T_0$.\\
The proof is similar for the Inverse topology.
\end{proof}

\begin{lem}
The I-topology and the Z-topology agree on $\Ord_*(G)$.
\end{lem}

\begin{proof}
Indeed, for $u\in G$, $u\neq 1$ we have $\mathcal U_{u}\cap\Ord_*(G)=\mathcal O_{u}\cap\Ord_*(G).$
\end{proof}

\begin{rem}
Let $G$ be a group. We can equip $G$ with the cofinite topology: the closed sets are $G$ along with the finite subsets of $G$. Now, for $F\subset G$ finite and $U\subset G$ cofinite, we set
$$\mathcal C(F,U):=\{V \text{ preorder monoid  } \mid F\cap V=\emptyset \text{ and } U\cap V\neq \emptyset\}.$$
The \emph{Chabauty topology} on the set of preorder monoids  is the topology generated by the  $\mathcal C(F,U)$, when $F$ (resp. $U$) runs over the finite (resp. cofinite) subsets of $G$ (cf. \cite{Ch} or \cite[Definition I.10]{Wa}). Therefore, by identifying a preorder with its preorder monoid, we have
$$\mathcal C(F,U)=\bigcup_{u\in U}\O_u\cap\bigcap_{u\in F}\U_{u^{-1}}.$$
This shows that the Patch topology is the Chabauty topology on the set of preorder monoids (see also \cite{Na}).
\end{rem}

\subsection{Compactness of the space of preorders}

\begin{thm}\label{compacity_preorders}
Let $G$ be a group. Then $\ZR_\ast(G)$ is compact for the P-topology. Therefore it is compact for the Z-topology and the I-topology.
\end{thm}

\begin{proof}
We follow the method of Samuel and Zariski \cite[Theorem 40]{SZ}.\\
We do the proof for the space of left-invariant preorders. The case of bi-invariant preorders is similar.\\
For every $\preceq \in \ZR_l(G)$, we define the map $\nu_\preceq \colon G \to \{-1,0,1\}$ as follows:
$$\nu_\preceq(u):=\left\{
    \begin{array}{lll}
        -1 & \mbox{if } u \prec 1 \\
        \mbox{ }0  &\mbox{if } u \sim_\preceq 1 \\
        \mbox{ }1  &\mbox{if } u\succ 1.
    \end{array}
\right.$$
This defines an inclusion
$$\ZR_l(G) \subset \{-1,0,1\}^{G }.$$
We consider the discrete topology on $\{-1,0,1\}$, and we consider the product topology on $\{-1,0,1\}^{G }$. The induced topology on $\ZR_l(G)$ is the P-topology.\\
We have that $\{-1,0,1\}$ is compact, and the product $\{-1,0,1\}^{G}$ is compact by Tychonoff's Theorem. In the corresponding product topology, we claim that $\ZR_l(G)$ is a closed set, so  compact. That is, $\ZR_l(G)$ is compact in the P-topology.\\\\
Thus let us prove that $\ZR_l(G)$ is closed in $\{-1,0,1\}^{G}$.
For any map $\nu\in \{-1,0,1\}^{G}$, we have that $\nu \in \ZR_l(G)$ if and only if:
\begin{enumerate}
\item[i)] For all $u, v \in G$, either $\nu(u)=-1$ or $\nu(v)=-1$ or $\nu(uv)\in \{0,1\}$
\item[ii)] For all $u \in G$, $\nu(u^{-1})=-\nu(u)$.
\end{enumerate}
Clearly, if $\nu=\nu_\preceq$ for some $\preceq\in\ZR_*(G)$, these properties are satisfied. On the other if $\nu$ satisfies these properties, let us show that $\nu=\nu_\preceq$ for some $\preceq$. In this case, necessarily $\preceq$ is defined  as follows: $\forall u, v\in G$, 
$u \prec v$ if $\nu(u^{-1}v)=1$; $u \succ v$ if $\nu(u^{-1}v)=-1$;  $u\sim_\preceq v$ if $\nu(v^{-1}u)=0$. We only need to prove that $\preceq$ is a preorder in the sense of Definition \ref{preorder}. \\
Clearly, for every $u$, $v\in G$ we have $u\preceq v$ or $u\succeq v$. By ii), we have $u\sim_\preceq u$ for every $u\in G$. Then let $u$, $v$, $w$ be elements of $G$, and assume $u\preceq v$ and $v \preceq w$. It means that $\nu(u^{-1}v)$ and $\nu(v^{-1}w)$ are in $\{0,1\}$. By i), we have $\nu(u^{-1}w)\in \{0,1\}$, hence $u \preceq w$.\\
Now let $u$, $v$, $w\in G$ with  $v\preceq w$, that is, $\nu(v^{-1}u^{-1}uw)=\nu(v^{-1}w)\in \{0,1\}$. Thus, $uv \preceq uw$.\\
For every $u \in G$, we denote by $\Phi_{u} \colon \{-1,0,1\}^{G} \to \{-1,0,1\}$ the map sending $\nu$ onto $\nu(u)$. This map is continuous for the discrete topology.\\
For every $u$, $v \in G$, we set 
 $$F_{u,v}:=\Phi_{u}^{-1}\left(\{-1\}\right)\cup \Phi_{v}^{-1}\left(\{-1\}\right)\cup \Phi_{uv}^{-1}\left(\{0,1\}\right)$$
 and
 $$F'_{u}:=\Phi_{u}^{-1}\left(\{0\}\right)\cup \left( \Phi_{u}^{-1}\left(\{1\}\right)\cap \Phi_{u^{-1}}^{-1}\left(\{-1\}\right)\right)$$
  Since $\Phi_{u}$ is continuous for the discrete topology, the sets $F_{u,v}$ and $F'_{u}$ are closed sets for all $u$, $v \in G$.\\
Moreover, we have that $\ZR_l(G)= \bigcap\limits_{u,v \in G} F_{u,v}\cap \bigcap\limits_{u\in G} F'_{u}.$
Therefore $\ZR_l(G)$ is a closed set. This completes the proof.
\end{proof}

\begin{lem}
Let $G$ be a group. Then $\Ord_*(G)$ is a closed set of $\ZR_*(G)$ in the Z-topology and the P-topology.
\end{lem}

\begin{proof}
We have $\Ord_*(G)=\bigcap\limits_{u \in G, u\neq 1}\left(\mathcal{O}_{u}\cap \mathcal{O}_{u^{-1}} \right)^c.$
\end{proof}

\begin{rem}
As a closed subset of a compact set, $\Ord_*(G)$ is compact for the Z-topology. Since the Z-topology and the I-topology coincide on the set of orders,  $\Ord_*(G)$ is compact for the I-topology.
\end{rem}

\begin{prop}
Let $G$ be a group. Then $\ZR(G)$ is a closed subset of $\ZR_l(G)$ for the I-topology and the P-topology.
\end{prop}

\begin{proof}
Let $\preceq\in \ZR_l(G)$. We have that $\preceq$ is bi-invariant if and only if it is right invariant, that is
$$\forall u,v,w \in G,\ \  u\preceq v\Longleftrightarrow uw\preceq vw.$$
Therefore $\ZR(G)=\bigcap_{u\in G}\left(\bigcap_{w\in G} \O_{w^{-1}uw}\cup\bigcap_{w\in G} \O_{w^{-1}u^{-1}w}\right)$
is closed for these two topologies.
\end{proof}
\subsection{Residue group of a preorder}

\begin{defi}
Let $G$ be a group and $\preceq\in\ZR_l(G)$. Let $H$ be a subset of $G$. We say that $H$ is \emph{$\preceq$-isolated} (or \emph{$\preceq$-convex}) if, for every $u_1$, $u_2\in H$, $v\in G$,
$$u_1\preceq v\preceq u_2\Longrightarrow v\in H.$$
\end{defi}

\begin{lem}\label{subgroup}
Let $G$ be a  group and let $\preceq \in \ZR_l(G)$.
The set 
$$G_\preceq:=\{u\in G\mid u\sim_\preceq 1\}$$
is a $\preceq$-isolated  subgroup of $G$ called the \emph{residue group} of $\preceq$.\\
Moreover, if $\preceq$ is bi-invariant, then $G_\preceq$ is normal.
\end{lem}

\begin{proof}
It is straightforward to check that $G_\preceq$ is a subgroup. Let us prove that $G_\preceq$ is $\preceq$-isolated. Let $u$, $v\in G_\preceq$ and $w\in G$ such that
$u\preceq w\preceq v$.
Then 
$1\preceq u \preceq w\preceq v\preceq 1$,
hence $w\in G_\preceq$. Thus $G_\preceq$ is $\preceq$-isolated.\\
Let us prove that $G_\preceq$ is normal when $\preceq$ is bi-invariant.
Let $u\in G$ and $v\in G_\preceq$. Then
$uv \sim_\preceq u$, thus $uvu^{-1}\sim_\preceq 1$. Thus $uG_\preceq u^{-1}\subset G_\preceq$ for every $u\in G$. Hence $G_\preceq$ is normal.
\end{proof}

\begin{rem}
Equivalently, we have $G_\preceq=V_\preceq\setminus \m_\preceq$.
\end{rem}

\begin{lem}\label{inc_subgroups}
Let $\preceq  \,\boldsymbol{\leq}\, \preceq'$ be two elements of $\ZR_*(G)$ where $G$ is a group. Then 
$$G_{\preceq'}\subset G_\preceq$$
with equality if and only if $\preceq=\preceq'$.
\end{lem}
\begin{proof}
Let $u\in G_{\preceq'}$, that is, $u\preceq'1$ and $1\preceq' u$. Since $\preceq'$ refines $\preceq$ we have $u\preceq 1$ and $1\preceq u$, that is, $u\in G_\preceq$.
\end{proof}

\begin{cor}\label{tree2}
If $G$ is a Noetherian and Artinian group, then $\ZR_l(G)$ is a rooted tree.\\
If $G$ is a group  satisfying the ascending and descending chain conditions on normal subgroups, then $\ZR(G)$ is a rooted tree.
\end{cor}

\begin{proof}
This is a direct consequence of Proposition \ref{tree}, and Lemmas \ref{subgroup} and \ref{inc_subgroups}.
\end{proof}

\begin{prop}\label{quotient_subgroup}

Let $G$ be a group and $H$ be a normal subgroup of $G$. Then
 $$\Ord_*(G/H)\neq \emptyset\Longleftrightarrow \exists\preceq\in\ZR_*(G)\text{ such that }H=G_\preceq.$$
\end{prop}

\begin{proof}
Let $\preceq' \in \Ord_*(G/H)$.
This induces a preorder on $G$ by defining for every $u$, $v \in G$:
$$ u\preceq v \Longleftrightarrow \overline{u}\preceq' \overline{v}.$$
It is straightforward to check that $H=G_{\preceq}$.
\end{proof}

\begin{ex}\label{ex_ab_ord}
If $G$ is a torsion-free abelian group then $\Ord(G)\neq \emptyset$ (see Example 1.3.8 \cite{G}).
\end{ex}

\begin{prop}\label{exact_seq}
Let $G$ be a group.
Let $\preceq\in \ZR_*(G)$. 
Then we have:
\begin{itemize}
\item[i)] There is a increasing bijection between $\ZR_*(G_\preceq)$ and  $\Raf_*(\preceq)$. This bijection is a Z-homeomorphism and an I-homeomorphism.
\item[ii)] If $\preceq$ is bi-invariant, there is an injective increasing Z-continuous and I-continuous map 
$$\psi_\preceq:\Ra_*(\preceq)\lgw \ZR_*(G/G_\preceq).$$
Its image is $\Ra_*(\psi_\preceq(\preceq))$. Moreover the inverse
$$\psi_\preceq^{-1}:\Ra_*(\psi_\preceq(\preceq))\lgw \Ra_*(\preceq)$$
is increasing, Z-continuous, and I-continuous.
\end{itemize}
\end{prop}

\begin{proof}
Let us prove i). First we show that the inclusion $G_\preceq \subset G$ induces a bijection $\phi_\preceq$ between $\ZR_*(G_\preceq)$ and  $\Raf_*(\preceq)$. Let $\preceq'\in\ZR_*(G_\preceq)$. This preorder defines in a unique way a preorder $$\phi_\preceq(\preceq'):=\preceq''\in  \Raf_*(\preceq)$$ as follows:\\  
Let $u$, $v\in G$. 
If $ u\prec v$ then we set $u\prec'' v$.\\
 If $u\sim_\preceq v$  then $v^{-1}u\in G_\preceq$ and we set $u\preceq''v$ (resp. $u\succeq'' v$) if  $v^{-1}u\preceq' 1$ (resp. $v^{-1}u\succeq'1$).\\
It is straightforward to check that $\preceq''\in \ZR_*(G)$ refines $\preceq$ (that is $\preceq''\in  \Raf_*(\preceq)$), and that the restriction of  $\preceq''$ to $G_\preceq$ is $\preceq'$. Thus $\phi_\preceq$ is a bijection and its inverse is the restriction map:
$$\preceq''\in  \Raf_*(\preceq)\longmapsto \preceq''_{|G_\preceq}\in\ZR_*(G_\preceq).$$
For $u\in G_\preceq$, we have 
$$(\phi_\preceq^{-1})^{-1}(\mathcal O_{u}\cap\ZR_*(G_\preceq))=\mathcal O_{u}\cap   \Raf(\preceq).$$
Hence $\phi_\preceq^{-1}(\mathcal O_{u}\cap   \Raf(\preceq))=\mathcal O_{u}\cap\ZR_*(G_\preceq)$,
Therefore $\phi_\preceq$ and $\phi_\preceq^{-1}$ are Z-continuous and I-continuous. Moreover these two maps are increasing maps from their construction.\\
\\
Now let us prove ii). Let $\preceq'\in\ZR_*(G)$ such that $ \preceq'\ \boldsymbol{\leq} \ \preceq$. Then $G_\preceq\subset G_{\preceq'}$. Therefore $\preceq'$ induces a preorder $\preceq''$ on $G/G_\preceq$ by defining
$$\ovl u\preceq'' 1\Longleftrightarrow u\preceq' 1$$
for every $u\in G$, where $\ovl u$ denotes the image of $u$  in $G/G_\preceq$. Then $\preceq''$ is well defined because, if $v\in G$ is such that $\ovl u=\ovl v$, we have $v^{-1}u\in G_\preceq\subset G_{\preceq'}$ and $v\preceq' 1$ when $u\preceq' 1$.
Thus we can define a map 
$$\psi_\preceq: \Ra_*(\preceq)\lgw \ZR_*(G/G_\preceq)$$
such that $\psi_\preceq (\preceq')=\preceq''$, and this map is clearly injective and increasing. The image of $\preceq$ by $\psi_\preceq$ is an order on $G/G_\preceq$ and the image of $\psi_\preceq$ is included in $\Ra(\psi_\preceq(\preceq))$.\\
The inverse of $\psi_\preceq$ is defined by
$$u\ \psi_\preceq^{-1}(\preceq'')\ v\Longleftrightarrow\ovl u\preceq''\ovl v$$
for every $u$, $v\in G$.\\
Now let $u\in G$. We have
$$\psi^{-1}_\preceq(\mathcal O_{\ovl u}\cap\Ra_*(\psi_\preceq(\preceq)))=\mathcal O_{u}\cap \Ra_*(\preceq)$$
and
$$(\psi^{-1}_\preceq)^{-1}\left(\mathcal O_{u}\cap \Ra_*(\preceq)\right)=\mathcal O_{\ovl u}\cap \Ra_*(\psi_\preceq(\preceq)).$$
Therefore $\psi_\preceq$ is  Z-continuous and I-continuous and $\psi_\preceq^{-1}$ also.
\end{proof}

\begin{prop}\label{quotient_map}
Let $G$ be a group and $H$ be a normal subgroup of $G$. Then there is a bijection, which is an increasing   Z-homeomorphism and a I-homeomorphism:
$$\ZR_*(G/H)\simeq \{\preceq\in\ZR_*(G)\mid H\subset G_\preceq\}.$$
\end{prop}

\begin{proof}
Let $\preceq\in\ZR_*(G)$ such that $ H \subset G_{\preceq}$. Therefore $\preceq$ induces a preorder $\preceq'$ on $G/H$ by defining
$$\ovl u\preceq' 1\Longleftrightarrow u\preceq 1$$
for every $u\in G$, where $\ovl u$ denotes the image of $a$  in $G/H$. Then $\preceq'$ is clearly well defined.
Thus we can define a map 
$$\psi_H: \{\preceq\in\ZR_*(G)\mid H\subset G_\preceq\}\lgw \ZR_*(G/H)$$
such that $\psi_H (\preceq)=\preceq'$, and this map is clearly injective and increasing.\\
The inverse of $\psi_H$ is defined by
$$u\ \psi_H^{-1}(\preceq')\ v\Longleftrightarrow\ovl u\preceq'\ovl v$$
for every $u$, $v\in G$. As in the proof of Theorem \ref{exact_seq} ii), it is straightforward to check that $\psi_H$ is  Z-continuous and I-continuous and $\psi_H^{-1}$ also.
\end{proof}

\begin{defi}\label{ZR_rel}
Let $G$ be a group, and $H$ be a subset of $G$. The relative Zariski-Riemann space $\ZR_*(G/H)$ is defined to be the set of $\preceq\in\ZR_*(G)$ such that $u\sim_\preceq 1$ for every $u\in H$. Equivalently,
$$\ZR_*(G/H)= \{\preceq\in\ZR_*(G)\mid H\subset G_\preceq\}.$$
\end{defi}

\begin{rem}
If $H$ is a normal subgroup of $G$, Proposition \ref{quotient_map} allows us this abuse of notation.
\end{rem}

\begin{lem}
Let $G$ be a group and $H$ be a subset of $G$. Then
\begin{enumerate}
\item[i)] $\ZR_l(G/H)=\ZR_l(G/\langle H\rangle)$ where $\langle H\rangle$ is the subgroup of $G$ generated by $H$.
\item[ii] $\ZR(G/H)=\ZR(G/\langle H\rangle^N)$ where $\langle H\rangle^N$ is the normal subgroup of $G$ generated by $H$.
\end{enumerate}
\end{lem}

\begin{proof}
This is a direct consequence of Lemma \ref{subgroup} and Definition \ref{ZR_rel}.
\end{proof}

\begin{ex}
Let $G$ be a group. The abelianization $G^{ab}$ of $G$ is the quotient of $G$ by its commutator subgroup: $G^{ab}=G/[G,G]$. By the previous proposition $\ZR(G^{ab})$ embeds in $\ZR(G)$. For instance if $G=F_n$, the free group generated by $n$ elements, we have that $\ZR(\Z^n)$ embeds in $\ZR(F_n)$.
\end{ex}

\begin{defi}\label{lex_order}
Let $G$ be group and $H$ a subgroup of $G$. We denote by $G/H$ the set of left cosets $\{uH\}_{u\in G}$. We consider the set of left invariants preorders on $G/H$:
$$X:=\{\preceq \text{ preorder on } G/H\mid \forall u,v,w\in G,\ uH\preceq vH\Longrightarrow wuH\preceq wvH\}.$$
For $\preceq\in X$, we can define $\preceq'\in\ZR_l(G)$ by
$$\forall u,v\in G,\ u\preceq'v \text{ if } uH\preceq vH.$$
Clearly, for every $w\in H$, $w\sim_{\preceq'} 1$, thus $\preceq'\in\ZR_l(G/H)$.\\
On the other hand, if $\preceq'\in \ZR_l(G/H)$, we define $\preceq\in X$ as
$$\forall u,v\in G,\ uH\preceq vH\text{ if } u\preceq' v.$$
Then $\preceq$ is well defined, since if $u$, $v\in G$ satisfy $u\preceq' v$, we have, for $w_1$, $w_2\in H$:
$$w_1\preceq'1, w_2\succeq' 1\Longrightarrow uw_1\preceq' u\preceq' v\preceq' vw_2.$$
Therefore we can identify $X$ with $\ZR_l(G/H)$, and we denote $X$ by $\ZR_l(G/H)$. The set of orders of $\ZR_l(G/H)$ is denoted by $\Ord_l(G/H)$.
\end{defi}

\begin{defi}\label{def_isolated}
Let $G$ be a group and $H$ a subgroup of $G$. Let $\preceq_1\in \Ord_l(G/H)$ and $\preceq_2\in \ZR_l(H)$. We define $\preceq\in\ZR_l(G)$ as follows:
$$\forall u,v\in G,\ u\preceq v\text{ if }\left\{\begin{array}{c} uH\prec_1 vH\\
\text{ or } uH\sim_{\preceq_1} vH\  (\text{i.e. }v^{-1}u\in H)\text{ and } v^{-1}u\preceq_2 1\end{array}\right.$$
We denote $\preceq$ by $\preceq_1\circ\preceq_2$ and it is called the \emph{composition of $\preceq_1$ and $\preceq_2$}. It is straightforward to see that $H$ is $\preceq$-isolated and $\preceq_1\LEQ\preceq$.

\end{defi}

\begin{lem}\label{isolated}
Let $\preceq$, $\preceq'\in\ZR_l(G)$, $\preceq'\LEQ\preceq$. Then $G_{\preceq'}$ is $\preceq$-isolated.
\end{lem}

\begin{proof}
Let $u\in G$, $v\in G_{\preceq'}$ satisfy $v\succeq u\succeq 1$. Then $v\succeq' u\succeq'1$. So $u\sim_{\preceq'}1$ and $u\in G_{\preceq'}$.
\end{proof}

\begin{prop}\label{lex_order2}
Let $G$ be a group and $H$ a subgroup of $G$. Let $\preceq\in\ZR_l(G)$ such that $H$ is $\preceq$-isolated. We define $\preceq_1$ by
$$\forall u,v\in G, \ uH\preceq_1 vH \text{ if }\left\{\begin{array}{c} v^{-1}u\in H\\
\text{ or } v^{-1}u\preceq 1\end{array}\right.$$
Then $\preceq_1$ is well defined and belongs to $\Ord_l(G/H)$. If we set $\preceq_2=\preceq_{|_H}$, we have
$$\preceq=\preceq_1\circ\preceq_2.$$
\end{prop}

\begin{proof}
If $v^{-1}u\in H$, then $w_2^{-1}v^{-1}uw_1\in H$,  for every $w_1$, $w_2\in H$. Therefore $uH\preceq_1 vH$ is well defined in this case.\\
Assume now that $v^{-1}u\notin H$ and $v^{-1}u\preceq 1$. Since  $H$ is $\preceq$-isolated,  we have $v^{-1}u\preceq w_2$, thus
$u^{-1}vw_2\succeq 1$. Again, because $H$ is $\preceq$-isolated, we have $u^{-1}vw_2\succeq w_1$, thus $vw_2\succeq uw_1$.  This proves that $\preceq_1$ is well defined.\\
It is straightforward to see that $\preceq_1$ is a preorder and that $\preceq_{1|_{H}}$ is trivial. Moreover, by definition, $\preceq_1$ is left invariant. Finally, if $v^{-1}u\sim_\preceq 1$, then $v^{-1}u\in H$ since $H$ is $\preceq$-isolated. Therefore $\preceq_1\in\Ord_l(G/H)$.\\
Now let $u$, $v\in G$ with $u\preceq v$. In particular, $uH\preceq_1 vH$. If $uH\sim_{\preceq_1} vH$, then $v^{-1}u\in H$, and $v^{-1}u\preceq_2 1$. This shows that $\preceq=\preceq_1\circ\preceq_2$.
\end{proof}

\begin{cor}
Let $G$ be a group and let $\preceq'$, $\preceq\in\ZR_l(G)$.  Then
$$\preceq'\LEQ\preceq\  \Longleftrightarrow \ \exists \preceq_2\in \ZR_l(G_{\preceq'}),\ \ \preceq=\preceq'\circ\preceq_2$$
\end{cor}
\begin{proof}
 Let $H:=G_{\preceq'}$.  It is straightforward to check that $\preceq'$ is equal to the preorder $\preceq_1$ defined in Proposition \ref{lex_order2}. Therefore there is $\preceq_2\in \ZR_l(G_{\preceq'})$ such that $\preceq=\preceq'\circ\preceq_2$ (just take $\preceq_2:=\preceq_{|_{G_{\preceq'}}}$).
\end{proof}

\subsection{Extension and restriction of preorders}

\begin{lem}
Let $G$ be a group. Then there is a bijection between $\ZR_*(G)$ and $\ZR_*(G/\text{Tor}(G))$. This bijection is an increasing Z-homeomorphism and a I-homeomorphism. 
\end{lem}

\begin{proof}
Let $u\in \Tor(G)$ be an element of  order $n$. Let $\preceq\in\ZR_*(G)$. Then if $1\preceq u$ we have
$$1\preceq u\preceq u^2\preceq\cdots\preceq u^n=1.$$
Thus $\Tor(G)\subset G_{\preceq}$ for every $\preceq\in \ZR_*(G)$. Therefore this lemma is a particular case of Proposition \ref{quotient_map} since where $H=\Tor(G)$.
\end{proof}

\begin{lem}\label{extension_Q_ev}
Let $G$ be an abelian   group. Then the restriction map
$$\ZR(\Q\otimes G)\lgm \ZR(G)$$
is an increasing Z-homeomorphism and I-homeomorphism.
\end{lem}

\begin{proof}
By the previous lemma we may assume that $G$ is torsion-free. Thus $G$ can be seen as a subgroup of $\Q\otimes G$ through the map $u\in G\lgm 1\otimes u\in\Q\otimes G$.\\
We define the map
$$\phi:\ZR(G)\lgw \ZR(\Q\otimes G)$$
by
$$\forall\  \left(\frac{n}{m}\otimes u,\frac{p}{m}\otimes v \right) \in \left(\Q\otimes G \right)^2, \forall \preceq \in \ZR(G), \frac{n}{m}\otimes u\ \phi(\preceq)\ \frac{p}{m}\otimes v \Longleftrightarrow nu\preceq pv.$$
It is bijective since its inverse is the restriction map
$$\preceq\in\ZR(\Q\otimes G)\lgw \preceq_{|G}\in\ZR(G).$$
We have
$\phi^{-1}(\mathcal O_{\frac{n}{m}\otimes u})=\mathcal O_{nu}$
and $(\phi^{-1})^{-1}(\mathcal O_{u})=\mathcal O_{1\otimes u}$
for every $u\in G$, $m$, $n\in\Z$.\\
It is straightforward to check that $\phi$ is increasing.
\end{proof}




\subsection{Rank and degree of a preorder}

\begin{defi}
Let $G$ be a group and $\preceq\in \ZR_*(G)$. We denote by $\#\Ra_*(\preceq)$ the cardinal of $\Ra_*(\preceq)$ (as an initial ordinal). We define the rank of $\preceq$ in $\ZR_*(G)$ to be 
$$\rk_*(\preceq):=\left\{\begin{array}{cc}\displaystyle\#\Ra_*(\preceq)-1 & \text{ if this cardinal is finite} \\
\displaystyle\#\Ra_*(\preceq) & \text{ if this cardinal is infinite}  \end{array}\right.$$
The subset of $\ZR_*(G)$ of preorders of rank equal to $r$ (resp. greater or equal to $r$) is denoted by $\ZR^r_*(G)$ (resp. $\ZR^{\geq r}_*(G)$). 

\end{defi}

\begin{defi}
Let $G$ be a group and $\preceq\in \ZR_*(G)$. The degree of $\preceq$ in $\ZR_*(G)$ is 
$$\deg_*(\preceq):=\left\{\begin{array}{cc}\displaystyle\sup_{\preceq'\in\Ord_*(G))\cap\Raf_*(\preceq)}\#\left(\Ra_*(\preceq')\setminus\Ra_*(\preceq)\right)-1 & \text{ if this supremum is finite} \\
\displaystyle\sup_{\preceq'\in\Ord_*(G))\cap\Raf_*(\preceq)}\#\left(\Ra_*(\preceq')\setminus\Ra_*(\preceq)\right) & \text{ if this supremum is infinite}  \end{array}\right.$$
The subset of $\ZR_*(G)$ of preorders of degree  equal to $d$ (resp. less or equal to $d$) is denoted by $^d\!\ZR_*(G)$ (resp. $^{\leq d}\!ZR_*(G)$).

\end{defi}

\begin{rem}
By Definition \ref{def_isolated} and Proposition \ref{lex_order2}, $\Ra_l(\preceq)$ is in bijection with the set of $\preceq$-isolated subgroups of $G$. By Lemma \ref{inc_subgroups}, this bijection is an increasing map. In particular the set
$$\{H\subset G\mid H \preceq\text{-isolated subgroup}\}$$
is totally ordered under inclusion. 
\end{rem}

\subsection{Action on the set of preorders}

\begin{defi}
Let $G$ be a group and let $\Aut(G)$ be the group of automorphisms of $G$. Then there is a left action of $\Aut(G)$  on $\ZR_*(G)$ defined as follows:
$$ (\phi,\preceq)\in \Aut(G)\times\ZR_*(G)\lgm\  \preceq_\phi=:\a_\phi(\preceq)$$
defined by
$$ \forall u,v\in G,\ u  \preceq_\phi v \text{ if } \phi(u)\preceq \phi(v).$$
\end{defi}

\begin{rem}
Let $G$ be a group. In fact the action of $\Inn(G)$ on $\ZR(G)$ is trivial. Therefore, if we consider only bi-invariants preorders, the previous action induces an action of the outer automorphisms group $\Out(G)$ on $\ZR(G)$.
\end{rem}

\begin{lem}\label{rank_deg_action}
The rank and the degree are preserved by this action.
\end{lem}

\begin{proof}
Let $\preceq \in \ZR_*(G)$ and $\phi\in\Aut(G)$. We consider the map :
$${\a_\phi}_{|_{\Ra_*(\preceq)}}:\preceq'\in \Ra_*(\preceq)\lgm \, \preceq'_\phi.$$
It is enough to show that the image of  ${\a_\phi}_{|_{\Ra_*(\preceq)}}$ is $\Ra_*(\preceq_\phi)$ and that ${\a_\phi}_{|_{\Ra_*(\preceq)}}$ is injective. \\
Let $\preceq' \in \Ra_*(\preceq)$, and $u,v \in G$ such that $u\preceq_\phi v$, then $\phi(u) \preceq \phi(v)$ hence $\phi(u) \preceq' \phi(v)$ since $\preceq' \LEQ \preceq$. It means that $u\preceq'_\phi v$ and the image of ${\a_\phi}_{|_{\Ra_*(\preceq)}}$ is in $\Ra_*(\preceq_\phi)$.\\
Now let $\preceq'$ be an element of $\Ra_*(\preceq_\phi)$. We can see that $\preceq'_{\phi^{-1}} \in \Ra_*(\preceq)$. Therefore the image of ${\a_\phi}_{|_{\Ra_*(\preceq)}}$ is $\Ra_*(\preceq_\phi)$.\\
Let $\preceq_1$ and $\preceq_2$ be two elements of $\Ra_*(\preceq)$. We can assume $\preceq_1 \LEQ \preceq_2$. Then we have $\preceq_{1\phi} \LEQ \preceq_{2\phi}$ and they are equal if and only if $\preceq_1=\preceq_2$, hence ${\a_\phi}_{|_{\Ra_*(\preceq)}}$ is injective.
\end{proof}


\begin{lem}
For every $\phi\in \Aut(G)$, the map 
$$\a_\phi:\ZR_*(G)\lgw \ZR_*(G)$$
$$\preceq\,\lgm \, \preceq_\phi$$
is an increasing continuous map for the Z-topology, the I-topology, and the P-topology.
\end{lem}

\begin{proof}
Let $\preceq \in \ZR_*(G)$ and let $u \in G$ such that $\preceq \in \mathcal{U}_{u}\subset \mathcal{O}_{u}$. Then $\a_\phi(\mathcal{U}_{u})= \mathcal{U}_{\phi^{-1}(u)}$ and $\a_\phi(\mathcal{O}_{u})= \mathcal{O}_{\phi^{-1}(u)}$. So the map $\a_\phi$ is continuous for the Z-topology, the I-topology and the P-topology.\\
It is an increasing map since the image of ${\a_\phi}_{|_{\Ra_*(\preceq)}}$ is $\Ra_*(\preceq_\phi)$, as shown in the proof of Lemma \ref{rank_deg_action}.
\end{proof}

\begin{ex}
The action of $\Aut(G)$ on $\ZR_*(G)$ is not faithful in general. For instance let us consider $G=\Q$ and $\phi\in\Aut(\Q)$ be defined by
$\phi(x)=2x$.
Then $\preceq_\phi=\preceq$ for every $\preceq\in \ZR(G)$.\\
This also shows that the action is not free.
\end{ex}


\begin{lem}\label{lemma_stab}
We denote the stabilizer of $\preceq\in\ZR_*(G)$ by $\Aut(G)_\preceq$. We have
\begin{itemize}
\item[i)] $\forall \phi\in\Aut(G)_\preceq$, $\phi(G_\preceq)= G_\preceq$,
\item[ii)] If $G_\preceq$ is normal then
$$\Aut(G)_\preceq=\{\phi\in\Aut(G) \mid \phi(G_\preceq)=G_\preceq \text{ and }\phi_{|_{G/G_\preceq}}\in\Aut(G/G_\preceq)_\preceq\}.$$
\end{itemize}
\end{lem}

\begin{proof}
Let $\phi\in\Aut(G)_\preceq$ and $u\in G_\preceq$. We have
$$1\preceq u \text{ and } u\preceq 1.$$
Since $\preceq_\phi=\preceq$ we have
$$1\preceq \phi(u) \text{ and } \phi(u)\preceq 1$$
that is, $\phi(u)\in G_\preceq$. By replacing $\phi$ by $\phi^{-1}$ we prove i).\\
\\
Now let $\phi\in \Aut(G)$ be such that $\phi(G_\preceq)=G_\preceq$ and $\phi_{|_{G/G_\preceq}}\in\Aut(G/G_\preceq)_\preceq$. Let $u\in G$ such that $u\preceq 1$.\\
 If $1\preceq u$, $u\in G_\preceq$. Since $\phi(G_\preceq)=G_\preceq$ we have $\phi(u)\preceq 1$.\\
 If $1\prec u$, the class of $u$ in $G/G_\preceq$, denoted by $\ovl u$, is not trivial. The preorder $\preceq$ induces a preorder on $G_\preceq$ by Proposition \ref{exact_seq}, that we still denote by $\preceq$. Therefore we have $1\prec \ovl u$. Since $\phi_{|G/G_\preceq}\in\Aut(G/G_\preceq)_\preceq$, we have
 $1\prec \phi_{|G/G_\preceq}(\ovl u)=\ovl{\phi(u)}.$
 Hence $1\prec_\phi u$. Therefore $\preceq_\phi=\preceq$ and $\phi\in\Aut(G)_\preceq$.\\
 The reverse inclusion is straightforward to check.

\end{proof}


\subsection{A metric: the case of countable groups}

In the case of a countable group $G$, Sikora \cite{S} proved that the Zariski topology on $\Ord_*(G)$ is a metric topology. We extend here this result to $\ZR_*(G)$ endowed with the Patch topology.

\begin{defi}\label{height}
Let $G$ be a countable group. Let
$$\mathcal G_1\subset \mathcal G_2\subset \cdots\subset \mathcal G_n\subset\cdots $$
be a chain  of finite subsets of $G$ such that $\bigcup\limits_{n\geq 1}\mathcal G_n=G$. We denote this chain by $\mathcal G$.
For a given $u\in G$,  we define the height of $u$ with respect to $\mathcal G$ as
$$\he_{\mathcal G}(u)=\min\{n\in\N^*\mid u\in \mathcal G_n\}.$$
\end{defi}

\begin{defi}
Let $\preceq_1$, $\preceq_2\in\ZR_*(G)$ where $G$ is a countable group.  Let us fix a chain $\mathcal G$ as in Definition \ref{height}. If $\preceq_1\neq\preceq_2$ we set
$$d_{\mathcal G}(\preceq_1,\preceq_2)=\frac{1}{n}$$ if $\preceq_{1|_{\mathcal G_{n-1}}}=\preceq_{2|_{\mathcal G_{n-1}}}$ but $\preceq_{1|_{\mathcal G_{n}}}\neq\preceq_{2|_{\mathcal G_{n}}}$. And we set $d_{\mathcal G}(\preceq_1,\preceq_1)=0$.
\end{defi}

\begin{prop}\label{prop_metric}
Let $G$ be a countable group and $\mathcal G$ be a chain as in Definition \ref{height}. Then we have:
\begin{itemize}
\item[i)] The function $d_{\mathcal G}$ is an ultrametric. 
\item[ii)] The topology defined on $\ZR_*(G)$ by  $d_{\mathcal G}$ is   the  Patch topology. In particular, it does not depend on the choice of $\mathcal G$.
\end{itemize}
\end{prop}

\begin{proof}
Clearly $d_{\mathcal G}$ is non negative, reflexive and symmetric. The ultrametric inequality is straightforward to check.
Therefore we only need to prove ii).\\
\\
Now let $n\in\N^*$ and $\preceq\in \ZR_*(G)$. We denote by $B\left(\preceq, \frac{1}{n}\right)$ the open ball centered at $\preceq$ of radius $\frac{1}{n}$ for the metric $d_\mathcal G$. Then $\preceq'\in B\left(\preceq, \frac{1}{n}\right)$ if and only if 
$$\forall u,v\in \mathcal G_n, \ \ \left\{\begin{array}{c} u\preceq v\Longrightarrow u\preceq' v\\ u\prec v \Longrightarrow u\prec' v\end{array}\right.$$
Therefore we have
$$ B\left(\preceq, \frac{1}{n}\right)=\bigcap_{u\in \mathcal G_n, u\succ 1}\mathcal U_{u} \cap \bigcap_{u\in \mathcal G_n, u\succeq 1}\mathcal O_{u} $$
is open for the topology generated by the $\U_{u}$ and the $\O_{u}$.
 Indeed this intersection is finite since the $\mathcal G_n$ are finite. \\
 On the other hand, let $u\in G$. For $\preceq\in\mathcal U_{u}$ we have
$B\left(\preceq, \frac{1}{\he_{\mathcal G}(u)}\right)\subset \mathcal U_{u}$.
For $\preceq\in\mathcal O_{u}$ we have
$B\left(\preceq, \frac{1}{\he_{\mathcal G}(u)}\right)\subset \mathcal O_{u}$.
 Thus the $\mathcal U_{u}$ and the $\O_{u}$ are open for the topology induced by $d_{\mathcal G}$.

\end{proof}

\begin{rem}\label{conv2}
Let $G$ be a countable group and $\{\mathcal G_k\}_k$ be a chain as in Definition \ref{height}. Let $(\preceq_n)_{n\in\N}$ be a sequence of preorders on  $G$ that converges to  $\preceq\in\ZR_*(G)$ for the Patch topology. Then
$$\forall k\in\N,\ \exists N\in\N,\ \forall n\geq N,\ G_{\preceq_n}\cap \mathcal G_k=G_{\preceq}\cap\mathcal G_k.$$
\end{rem}

\subsection{Cantor sets}
Assume that $G$ is a countable group. Then $\ZR_*(G)$, endowed with the Patch topology, is a metrizable compact space. Moreover it is totally disconnected:
\begin{lem}\label{tot_disc}
Any subspace $E\subset \ZR_*(G)$ is totally disconnected for the Patch topology.
\end{lem}
\begin{proof}
Indeed, if $\preceq_1$, $\preceq_2\in E$, $\preceq_1\neq\preceq_2$, there is $u\in G$ such that $u\succeq_1 1$ and $1\succ_2 u$ (eventually after permutation of $\preceq_1$ and $\preceq_2$). Therefore $\preceq_1\in \O_{u}$ and $\preceq_2\in \U_{u^{-1}}$. But $\O_{u}\cap\U_{u^{-1}}=\emptyset$ and $\O_{u}\cup\U_{u^{-1}}=\ZR_*(G)$.\\
\end{proof}

Moreover any closed subset of $\ZR_*(G)$ is a also a metrizable totally disconnected compact space. Therefore, a closed subset of $\ZR_*(G)$ is a Cantor set if and only if it is a perfect space, that is,  it does not have isolated points.

There are several cases for which $\Ord_*(G)$ is known to be a Cantor set. Here is a non complete list of some examples:
\begin{itemize}
\item The space $\Ord(\Q^n)$ for $n\geq 2$ is a Cantor set \cite{S}.
\item The space $\Ord_l(F_n)$ for $n\geq 2$ is a Cantor set, where $F_n$ is the free group generated by $n$ elements \cite{MC}, \cite{N}.
\item The space $\Ord_*(G)$, where $G$ is a countable, torsion-free, nilpotent group which is not rank-1 abelian, is a Cantor set \cite{MW} and \cite{DNR}.
\item The space $\Ord_l(G)$, where $G$ is a compact hyperbolic surface group, is a Cantor set \cite{ABR}.
\end{itemize}
In  \ref{Q^n} we will see that $\ZR(\Q^n)$, for $\geq 2$, contains infinitely many Cantor subsets.


\section{Examples} \label{section exemples}
\subsection{The $\Q$-vector spaces} 
By Lemma \ref{extension_Q_ev}, in order to study $\ZR(G)$ for an abelian group, we only need to study $\ZR(\Q\otimes G)$. Therefore we may assume that $G$ is  a $\Q$-vector space.
We begin with  the following lemma:

\begin{lem}\label{Q-ev}
Let $H$ be a normal subgroup of a $\Q$-vector space $G$ such that $\Tor(G/H)=\{0\}$. Then $H$ is a $\Q$-subspace of $G$.
\end{lem}

\begin{proof}
Since $H$ is a subgroup, $H$ is stable by addition and by multiplication by integers. Therefore, we only have to prove that for every $h\in H$ and $n\in\N^*$, $\frac{1}{n}h\in H$.\\
Indeed, for such a $h$ and such a $n$,  the image $g$ of $\frac{1}{n}h$ in $G/H$ is a torsion element since $ng\equiv 0$ modulo $H$. Hence $\frac{1}{n}h\in H$ since $\Tor(G/H)=\{0\}$.
\end{proof}

On the other hand, every $\Q$-subspace of $G$ is a normal subgroup of $G$ with $\Tor(G/H)=\{0\}$.
Therefore, by Example \ref{ex_ab_ord}, the residue groups $G_\preceq$ of preorders $\preceq\in\ZR(G)$ are exactly the $\Q$-subspaces of $G$.

\begin{prop}\label{deg_residue}
Let $G$ be a $\Q$-vector space and $\preceq\in\ZR(G)$. Then
$$\deg(\preceq)=\dim_\Q(G_\preceq).$$
\end{prop}

\begin{proof}
By Lemma \ref{Q-ev} and Example \ref{ex_ab_ord}, the residue groups of preorders on $G$ are the subvector spaces of $G$. Therefore, by Lemma \ref{inc_subgroups} and Proposition \ref{quotient_subgroup}, we have 
$\dim_\Q(G_\preceq)=\deg(\preceq)$.
\end{proof}

We have the following inequality relating the rank and the degree of a preorder (this can be seen as the counterpart of the inequality concerning valuations given in Remark \ref{Abhyankar}):
\begin{cor}\label{rank-dim}
Let $G$ be a $\Q$-vector space and $\preceq\in \ZR(G)$. Then
$$\rk(\preceq)+\deg(\preceq)\leq \dim_\Q(G).$$
\end{cor}

\begin{proof}
By definition, the rank of $\preceq$ is the supremum of the length of chains of preorders $\{\preceq_i\}_{i\in I}$ (i.e. $I$ is totally ordered and $\preceq_i\SEQ \preceq_j$ for every $i<j$) such that $\preceq_i\SEQ \preceq$ for every $i\in I$. By  Lemma \ref{inc_subgroups} and Lemma \ref{isolated}, it is also bounded by the supremum of the lengths of   chains $\{H_i\}_{i\in I}$ of subvector spaces such that
$G_\preceq\subsetneq H_i$ for every $i\in I$. This proves the claim, since $\deg(\preceq)=\dim_\Q(G_\preceq)$.
\end{proof}

\begin{cor}\label{conv3}
Let $G$ be a countable $\Q$-vector space, $\preceq\in\ZR(G)$ and let $(\preceq_n)_{n\in\N}$ be a sequence of preorders on  $G$ that converges to  $\preceq$ for the Patch topology. If $\dim_\Q(G_\preceq)<\infty$,
then, for $n$ large enough,
$$G_{\preceq}\subset G_{\preceq_n}.$$
In any case we have
$$\liminf\deg(\preceq_n)\geq \deg(\preceq).$$
\end{cor}

\begin{proof}
Let $\{\mathcal G_k\}_{k\in\N}$ be a chain as in Definition \ref{height}. Let $(u_l)_{l\in L}$ be a $\Q$-basis of $G_\preceq$.\\
Assume that $L$ is finite, and let $k$ be large enough to insure that $u_l\in\mathcal G_{k}$ for every $l\in L$.
By Remark \ref{conv2}, and since $G_{\preceq_n}$ is a $\Q$-vector space for every $n$, we have
$G_\preceq\subset G_{\preceq_n}$ for $n$ large enough.\\
Now assume that $L$ is infinite. For every integer $k$, we set $W_k:=\Vect_\Q(G_\preceq\cap\mathcal G_k)$. Therefore $G_\preceq=\bigcup\limits_{k=1}^\infty W_k=\sup\limits_k W_k$. Thus
$$\deg(\preceq)=\lim_k\dim_\Q(W_k)=\sup_k\dim_\Q(W_k)$$
by Lemma \ref{deg_residue}. But, still by Remark \ref{conv2}, for any fixed $k$,  we have $W_k\subset G_{\preceq_n}$ for $n$ large enough. Thus $\liminf\deg(\preceq_n)\geq \deg(\preceq)$. 
\end{proof}

\begin{ex}\label{ex_deg}
Let $(u_n)\in(\Q^2)^\N$ be a sequence of vectors of norm 1 that converges to a vector $u$ with an irrational slope. We set $\preceq_n:=\leq_{u_n}$ and $\preceq:=\leq_u\in\ZR(\Z^2)$. Then $\deg(\preceq_n)=1$ for every $n$ and $\deg(\preceq)=0$.\\
On the other hand, if we set $\mathcal G_k=\{-k,\ldots, k\}^2$, with $k\in\N$, we have the following property:
$$\forall k\in\N,\ \exists N\in\N,\ \forall n\geq N,\ \ \preceq_{n|_{\mathcal G_k}}\equiv\preceq_{|_{\mathcal G_k}}.$$
Thus $(\preceq_n)$ converges to $\preceq$ for the Patch topology. This shows that the inequality in Corollary \ref{conv3} may be strict.
\end{ex}

Now we consider a totally ordered set $I$ and we denote by $G_I$ a $\Q$-vector space with a basis $\{e_i\}_{i\in I}$. \\
A   subset $A$ of $I$ is called a \emph{upper set} (or just an \emph{upset}) if it satisfies the following property:
$$\forall i\in A,\ \forall j\in I,\ i\leq j\Longrightarrow j\in A.$$
For a given upset $A$ we define the preorder $\leq_A$ as follows:\\
For $u=\sum\limits_{k\in I}u_ke_k$ and $v=\sum\limits_{k\in I} v_ke_k\in G_I$, we define
 $$u<_A v$$ if and only if there is $i_0\notin A$,  such that $u_j=v_j$ for all $j<i_0$, and $u_{i_0}<v_{i_0}$. Then $\leq_A$ is well defined because all but finitely many $u_k$ and $v_k$ are zero. Moreover the residue group of $\leq_A$ is the $\Q$-vector space generated by the $e_j$ for $j\in A$.\\
 
 \begin{lem}
 Let $A\subset B$ be two upsets of $I$. Then
 $$\leq_B\LEQ\leq_A.$$
 \end{lem}
 
 \begin{proof}
 Let $u$, $v\in G_I$ such that $u<_{B} v$. Then there is $j_0\notin B$ such that $u_k=v_k$ for all $k<j_0$ and $u_{j_0}<v_{j_0}$. Since $A\subset B$, 
 $j_0\notin A$ and $u<_A v$. Therefore we have $u\leq_A v$. Thus,  $\leq_B\LEQ\leq_A$.
 \end{proof}

Therefore the subset $E:=\{\leq_A\mid A\text{ upset  of } I\}$ is a totally ordered subset of $\ZR(G_I)$. The set $E':=\{\leq_{\{j\in I\mid j\geq i\}}\mid i\in I\}$ is a subset of $E$, which is totally ordered, order isomorphic to $I$. \\

\begin{ex}If $I=\R$, we obtain a totally ordered subset $E'$ of $\ZR(G_I)$, such that for every $\preceq_1$, $\preceq_2\in E'$, there exists $\preceq'\in E'$ with 
$$\preceq_1\, \boldsymbol{<}\, \preceq'\, \boldsymbol{<}\, \preceq_2.$$
The upsets of $\R$ have the form
$$]x,+\infty[ \text{ or }[x,+\infty[ \text{ for some }x\in\R.$$
Therefore $E$ is order isomorphic to the union of two copies of $\R$, denoted by $\R_1$ and $\R_2$. The elements of $\R_1$ are the open right segments of $\R$, and the elements of $\R_2$ are the closed right segments of $\R$. The order on $\R_1\cup\R_2$ is the inclusion. 
\end{ex}


\begin{lem}
The chain $E=\{\leq_A\}_{A\text{ upset  of } I}$ is a maximal chain.
\end{lem}

\begin{proof} Let $\preceq\in\ZR(G_I)$ such that,  for every upset  $A$, we have
$\preceq\LEQ\leq_A$ or  $\leq_A\LEQ\preceq.$
For a upset  $A$ we set
$H_A:=\Vect\{e_i\mid i\in A\}$.
Therefore, by Lemma \ref{inc_subgroups}, we have
$H_A\subset G_{I\preceq}$ or $G_{I\preceq}\subset H_A.$
By Lemma \ref{Q-ev}, $G_{I\preceq}$ is a $\Q$-subspace of $G_I$. We set
$$B:=\{i\in I\mid e_i\in G_{I\preceq}\}.$$
We claim that $B$ is an upset  of $I$. Indeed, if it is not the case, there exist $i$, $j\in I$ with $i<j$ and $i\in B$, $j\notin B$. Set
$$C:=\{k\in I\mid k\geq j\}.$$
Then $e_j\in H_C$ and $e_i\notin H_C$. By hypothesis we have $G_{I\preceq}\subset H_C$ or $H_C\subset G_{I\preceq}$. But this contradicts the fact that $e_i\in G_{I\preceq}$ and $e_j\notin G_{I\preceq}$. Hence $B$ is an upset.\\
\\
Now we claim that $H_B=G_{I\preceq}$. Indeed, assume that this is not the case. Then there exist $u\in G_{I\preceq}\setminus H_B$ such that $u$ is not a multiple of some $e_i$. We can choose $u$ of the form
$\displaystyle u=\sum_{k=1}^nu_{i_k}e_{i_k}$
where none of the $i_k$ is in $B$. Assume that $i_1<i_2<\cdots <i_n$ and $u_{i_1}\neq 0$. Set
$$A:=\{i\in I\mid i\geq i_1\}.$$
Then $A$ is the smallest upset  containing $B$ and $u$.  We have
$H_B\subset H_A$ or $H_A\subset H_B$
 by Lemma \ref{inc_subgroups}. But we have 
 $G_{I\preceq}\subset H_A$ or $H_A\subset G_{I\preceq}.$
 By definition of $B$ we do not have $H_A\subset G_{I\preceq}.$ Therefore we have $H_B\subsetneq G_{I\preceq}\subsetneq H_A$. Since $u$ is not a multiple of some $e_i$, there is an upset  $D$ with $B\subsetneq D\subsetneq A$. Therefore we do not have
 $G_{I\preceq}\subset H_D$,  neither  $H_D\subset G_{I\preceq}.$
 This contradicts the hypothesis. Therefore $G_{I\preceq}=H_B$ and $\preceq=\leq_B$.
\end{proof}

\begin{rem}
In particular if $I$ is not finite, $I$ contains a countable subset $I'$. But, for a well chosen order $I'$ can be identified with $\Q$. Therefore, the chain $\{\leq_A\}_{A \text{ upset  of } I'}$ is a totally ordered set of preorders on $\ZR(G_I)$ that is not well ordered. Therefore $\ZR(G_I)$ is not a graph.
\end{rem}


\subsection{Description of $\ZR(\Q^n)$} \label{Q^n}
We have the following result:

\begin{thm}\cite[Theorem 4]{R}\label{rob}
Let $\preceq\in\ZR(\Q^n)$. Then there exist an integer $s\geq 0$ and vectors $u_1$,\ldots, $u_s\in\R^n$ such that
$$\forall u,v\in \Q^n,\ u\preceq v \Longleftrightarrow (u\cdot u_1,\ldots, u\cdot u_s)\leq_{\text{lex}} (v\cdot u_1,\ldots, v\cdot u_s).$$
Then we write $\preceq=\leq_{\left(u_1,\dots,u_s\right)}$.
\end{thm}

\begin{prop}
For a given non trivial preorder $\preceq\in\Q^n$, let $s$ be the smallest integer $s$ satisfying Theorem \ref{rob}. Then the rank of $\preceq$ is $s$. 
\end{prop}

\begin{proof}
We have $\#\Ra(\preceq)=s+1$. Indeed $\Ra(\preceq)=\{\preceq_{\emptyset}, \leq_{u_1},\leq_{\left(u_1,u_2\right)},\dots,\leq_{\left(u_1,\dots,u_s\right)}\}$.
\end{proof}

%

\begin{prop} \label{deg_rank}
For a given non trivial preorder $\preceq\in\Q^n$ let $u_1$,\ldots, $u_s$ be vectors such that $\preceq=\leq_{(u_1,\ldots, u_s)}$.
We assume $s$ to be minimal for this property. For $k=0$, \ldots, $s$, let $\Psi_k$ be the  $\Q$-linear map
$$\begin{array}{cccc}\Psi_k : &  \Q^n& \lgw & \R^k\\
 & q& \lgm &(q\cdot u_1,\ldots, q\cdot u_k)\end{array}$$
 where $\Psi_0$ is the zero map.
Then $\Ker(\Psi_s)=G_\preceq$ and the following subgroups are the only $\preceq$-isolated subgroups of $\Q^n$:
$$G_\preceq=\Ker(\Psi_s)\subsetneq\Ker(\Psi_{s-1})\subsetneq\cdots\subsetneq \Ker(\Psi_1)\subsetneq\Ker(\Psi_0)= \Q^n.$$

\end{prop}

\begin{defi}
For such a preorder, we set 
$d_k:=\dim_\Q(\Ker(\Psi_{k-1})/\Ker(\Psi_k)).$
In particular we have 
\begin{equation}\label{ab_rr}\sum\limits_{k=1}^s{d_k}+\deg(\preceq)=n.\end{equation}
 Here, the integer $\sum\limits_{k=1}^s{d_k}$ is the analogue of the rational rank of a valuation, and \eqref{ab_rr} is the analogue of the second inequality in Remark \ref{Abhyankar}.
The sequence $(d_1,\ldots, d_s)$ is called the \emph{type} of the preorder.\\
The set of preorders of type $(d_1,\ldots, d_s)$ is denoted by $\ZR^{(d_1,\ldots, d_s)}(\Q^n)$.
\end{defi}

\begin{proof}[Proof of Proposition \ref{deg_rank}]
It is straightforward to see that $\Ker(\Psi_s)=G_\preceq$. \\
Moreover if $V_{k+1}=V_k$, then we have
$$\leq_{(u_1,\ldots, u_n)}=\leq_{(u_1,\ldots, u_{k-1},u_{k+1},\ldots, u_n)}.$$
Therefore, since $s$ is assumed to be minimal, we have $V_{k+1}\subsetneq V_k$ for every $k$.\\
We denote by $V_k$ the space $\Ker(\Psi_k)$. Let $k\geq 1$, and let $V$ be a subspace of $V_k$ such that $V_{k+1}\subsetneq V$. Assume that $V$ is $\preceq$-isolated. There is $v\in V$ such that $v\in V_k\setminus V_{k+1}$. Let $u\in V_k$. Since $v\notin V_{k+1}$, $v\cdot u_{k+1}\neq 0$. Thus there is  $m\in\Z$ such that
$$-m v\cdot u_{k+1}\leq u\cdot u_{k+1}\leq m v\cdot u_{k+1}.$$
Thus, $-mv\preceq u\preceq mv$. Therefore $u\in V$, since  $V$ is $\preceq$-isolated. Thus $V=V_k$. This proves the result.
\end{proof}

\begin{rem}\label{deg_int}
Assume given $u_1$, \ldots, $u_s$ as in Proposition \ref{rob}. Then we can replace $u_k$, for $k\geq 1$, by the orthogonal projection of $u_k$ onto $\R\otimes \Ker(\Psi_1)$. By induction, we may assume that $u_k\in\R\otimes\Ker(\Psi_{k-1})$ for every $k$.\\
 In this case, if we define $e_k=\dim_\Q(\Vect_\Q(u_{k,1},\ldots, u_{k,n}))$,  we have $\dim_\Q \Ker(\Psi_1)=n-e_1$ and $\R\otimes \Ker(\Psi_1)\simeq \R^{n-e_1}$. In particular $e_1=d_1$. Moreover
 $\Ker(\Psi_2)=\Ker(\Psi')$
 where $\Psi'$ is  the  $\Q$-linear map
$$\begin{array}{cccc}\Psi' : &  \Ker(\Psi_1) & \lgw & \R^{s-1}\\
 & q& \lgm &(q\cdot u_2,\ldots, q\cdot u_s)\end{array}$$
 Hence, by induction, we  have $e_k= d_k$ for every $k$.
\end{rem}

Then we have the following description of the topological set $\ZR(\Q^n)$:

\begin{thm}\label{main1} We fix $n\geq 2$ and consider $\ZR(\Q^n)$ endowed with the Patch topology. We have the following properties:
\begin{itemize}
\item[i)]  Every $\preceq\in\ZR(\Q^n)$ is an isolated point if and only if $\deg(\preceq)\geq n-1$. If $\deg(\preceq)\leq n-2$, every open neighborhood of $\preceq$ contains infinitely many preorders of same rank and same degree as $\preceq$.

 \item[ii)]  For $n\geq d\geq 0$, $^{\leq d}\!\ZR(\Q^n)$ is a metric compact totally disconnected space. Therefore, for $d\leq n-2$, it is a Cantor set.
 
 \item[iii)]   For $\preceq\in \ZR(\Q^n)$, the set $\Raf(\preceq)$ is homeomorphic to $\ZR(\Q^{\deg(\preceq)})$.

  \item[iv)]  Let  $\preceq\in \ZR(\Q^n)$ and $d\leq \deg(\preceq)-1$.
  Then  $\Raf(\preceq)\cap {}^{\leq d}\!\ZR(\Q^n)$ is a Cantor set.
 
 \item[v)] The only elements of $\Aut(\Q^n)$ whose action on $\ZR(\Q^n)$ is the identity, is the set of $\Q$-linear maps $x\lgm \la x$ with $\la\in\Q_{>0}$.
 
 \item[vi)] Let $\preceq\in\ZR(\Q^n)$. Then the stabilizer of $\preceq$ under the action of $\Aut(\Q^n)$ is 
 $$\Aut(\Q^n)_\preceq=\{\phi\in\Aut(\Q^n)\mid \phi(G_{\preceq})= G_{\preceq},  \text{ and } \phi_{|_{G/G_\preceq}}=\la 1\!\!1_{|_{G/G_\preceq}} \text{ with }\la>0\}.$$

\item[vii)]  For every $s$, $d$ and $(d_1,\ldots,d_s)\in\Z_{>0}^s$ with $\sum d_k+d=n$, $\Aut(\Q^n)$ acts  transitively on $\ZR^{(d_1,\ldots, d_s)}(\Q^n)$.

 \end{itemize}
 \end{thm}
 
 \begin{proof}
Let us prove i). Take a basis of $\Q^n$, $u_1$, \ldots, $u_n$. Then $\{\leq_\emptyset\}=\bigcap\limits_{i=1}^n(\O_{u_i}\cap\O_{-u_i})$. Therefore $\leq_\emptyset$ is an isolated point. Now  let $u\in\Q^n$, $u\neq 0$. Let $v_2$, \ldots, $v_n\in\Q^n$ be such that $(u,v_2,\ldots, v_n)$ is an orthogonal basis of $\Q^n$. Then we have
 $$\{\leq_u\}=\U_u\bigcap \left(\bigcap\limits_{i=2}^n \O_{v_i}\cap\O_{-v_i}\right).$$
 Therefore $\leq_u$ is an isolated point of $\ZR^1(\Q^n)$.\\
On the other hand, assume that $\preceq$ is not the trivial preorder nor a preorder of the form $\leq_u$ for some $u$ multiple of a vector in $\Q^n$. We set $s=\rk(\preceq)$ and $d=\deg(\preceq)$. If $s=1$, $\preceq=\leq_u$ for some $u\in\R^n$ that is not a multiple of a vector of $\Q^n$. Therefore, by Proposition \ref{deg_rank}, $d\leq n-2$. If $s\geq 2$, we have that $\deg(\preceq)=\dim_\Q(G_\preceq)\leq s-2\leq n-2$ by Corollary \ref{rank-dim}. Thus we always have $d\leq n-2$.\\
Assume that $\preceq$ is an isolated point and write $\preceq=\leq_{(u_1,\ldots, u_s)}$. Therefore we may assume that there are vectors $v_i$, $w_j\in\Q^n$ such that 
 $$\left(\bigcap\limits_{i=1}^r \O_{v_i}\right)\cap\left(\bigcap\limits_{j=1}^s\U_{w_j}\right)=\{\preceq\}.$$
 We may assume that $v_i\sim_\preceq 0$ for every $i$ and write
 $$E:=\left(\bigcap\limits_{i=1}^r \O_{v_i}\cap\O_{-v_i}\right)\cap\left(\bigcap\limits_{j=1}^s\U_{w_j}\right)=\{\preceq\}.$$ 
 We may also assume that none of the $v_i$ and $w_j$ are collinear.\\
Moreover $v_i\in G_\preceq$ for every $i$. Therefore we will show how to construct infinitely many preorders of rank $s$ and degree $d$ belonging to $E$, and contradicting the assumption on $\preceq$. For this we consider the set
 $$C:=\{x\in\R^n\mid x\cdot w_j> 0 \text{ for every } j\leq s\}.$$
This is a non-empty open set (since $\preceq\in \bigcap\limits_{j=1}^s\U_{w_j}$).  Therefore we may choose $u_1'$, \ldots, $u_{s}' \in C\cap G_\preceq^\perp$, $\Q$-linearly independent (because $\dim_\Q(G_\preceq^\perp)=n-d\geq s$ by Corollary \ref{rank-dim}). Moreover we may choose the $u_i'$ in such a way that the kernel of the  linear map $\Psi$ defined in Proposition \ref{deg_rank} is $d$. Indeed, by Remark \ref{deg_int}, in order to do this, we choose $u_1'$ such that 
$$d_1:=\dim_\Q(\Vect_\Q(u_{1,1}',\ldots, u_{1,n}'))=\dim_\Q(\Vect_\Q(u_{1,1},\ldots, u_{1,n})),$$
and by induction, we choose 
$$u_i'\in\R\otimes (\Vect(u_1,\ldots, u_{i-1})^\perp\cap\Q^n)$$
with 
$$d_i:=\dim_\Q(\Vect_\Q(u_{i,1}',\ldots, u_{i,n}'))=\dim_\Q(\Vect_\Q(u_{i,1},\ldots, u_{i,n})).$$
By Remark \ref{deg_int}, $\sum_i d_i= \deg(\preceq)$, thus we may choose such $u_i'$. And again by Remark \ref{deg_int},  the preorder 
$\leq_{(u_1',\ldots, u_s')}$
has degree $d$. Moreover it has rank $s$ since we have
$$\leq_{u_1'}\SEQ \leq_{(u_1',u_2')} \SEQ \cdots \SEQ \leq_{(u_1',\ldots, u_s')}.$$
 Moreover the residue group of $\leq_{(u_1',\ldots, u_s')}$ contains $G_\preceq$, because the $u_i'$ belong to $G_\preceq^\perp$. Since the $u_i'$ are in $C$ we have
$$\leq_{(u_1',\ldots, u_s')}\in E.$$
Because there are infinitely many ways of choosing the vector $u_1'$ of norm 1 (and therefore of choosing the unique preorder of rank 1 refined by $\leq_{(u_1',\ldots, u_s')}$), $E$ contains infinitely many preorders of rank $s$ and degree $d$. This proves i).\\
\\
The set $^{\leq d}\!\ZR(\Q^n)$ is closed by Corollary \ref{conv3}, therefore it is compact. This set is a metric space. Moreover it is totally disconnected, by Lemma \ref{tot_disc}. Therefore, by i), it is a Cantor set for $d\leq n-2$. \\
\\
Clearly iii) holds by  Proposition \ref{deg_rank} and Proposition \ref{exact_seq}.\\
 \\
We have that $\Raf(\preceq)\cap {}^{\leq d\!}\ZR(\Q^n)$ is homeomorphic to $^{\leq d}\!\ZR(\Q^{\deg(\preceq)})$ by Proposition \ref{exact_seq}. Therefore iv) follows from ii) and iii).\\
\\
Let $\phi$ be defined by $\phi(x)=\la x$, for every $x\in\Q^n$, where $\la>0$. Then, for $\preceq\in \ZR(\Q^n)$, and for every $u$, $v\in\Q^n$, we have $u\preceq v$ if and only if $\la u\preceq \la v$. That is, $\preceq=\preceq_\phi$. On the other hand, assume that $\phi$ is not of this form. Then there is $x\in\Q^n$ such that $x\neq \la\phi(x)$ for all $\la>0$. Thus, there is $u\in\Q^n$, such that $x\cdot u>0$ and $\phi(x)\cdot u<0$. Set $\preceq:=\leq_u$. Then we have $0\prec x$ but $0\succ \phi(x)$.
Therefore $\preceq\neq\preceq_\phi$. This proves v).\\
\\
Therefore, by Lemma \ref{lemma_stab} ii), we have vi).\\
\\
If $\preceq\in \ZR^{(d_1,\ldots, d_s)}(\Q^n)$ and $\phi\in\Aut(\Q^n)$, we have $\preceq_\phi\in \ZR^{(d_1,\ldots, d_s)}(\Q^n)$ by Proposition \ref{deg_rank}.\\
Let $\preceq$, $\preceq'\in \ZR^{(d_1,\ldots, d_s)}(\Q^n)$. We denote by 
$$V_s:=G_\preceq\subsetneq V_{s-1}\subsetneq\cdots \subsetneq V_1\subsetneq V_0:=\Q^n$$
$$(\text{resp. } V_s':=G_{\preceq'}\subsetneq V_{s-1}'\subsetneq\cdots \subsetneq V_1'\subsetneq V_0':= \Q^n)$$
the $\preceq$-isolated (resp. $\preceq'$-isolated) subvector spaces of $\Q^n$. Therefore $\dim_\Q(V_{k-1}/V_{k})=\dim_\Q(V'_{k-1}/V'_{k})=d_k$ for every $k$. If $\preceq=\leq_{(u_1,\ldots, u_s)}$ and $\preceq'=\leq_{(u'_1,\ldots, u'_s)}$, we have $V_1=\lg u_1\rg^\perp\cap\Q^n$ and $V'_1=\lg u'_1\rg\cap\Q^n$. After a $\Q$-linear change of coordinates, we may assume that 
$V_1=V'_1=\{0\}\times\Q^{d_1}$, in particular $u_1=(u_{1,1},\ldots, u_{1,n-d_1},0,\ldots, 0)$ and $u'_1=(u'_{1,1},\ldots, u'_{1,n-d_1},0,\ldots, 0)$.\\
We have $\dim_\Q(\sum\Q u_{1,i})=\dim_\Q(\sum\Q u'_{1,i})=n-d_1$, therefore there exists a $(n-d_1)\times(n- d_1)$-matrix $A$ with entries in $\Q$ such that
$A\begin{bmatrix}u_{1,1}\\ \cdots\\u_{1,n-d_1}\end{bmatrix}=\begin{bmatrix}u'_{1,1}\\ \cdots\\u'_{1n-,d_1}\end{bmatrix}$. Now we apply an induction on $s$, and assume the result is true for $s-1$, that is, there is a linear $\phi':\Q^{d_1}\lgw \Q^{d_1}$, whose matrix is denoted by $B$, such that
$$\leq_{(u_2,\ldots, u_s)_{\phi'}}=\leq_{(u_2',\ldots, u'_s)}.$$
Therefore we consider the $\Q$-linear map $\phi$ whose matrix is $\begin{bmatrix} A & 0\\ 0 & B\end{bmatrix}$, and we have $\preceq_\phi=\preceq'$. This proves vii).
  \end{proof}

  \begin{ex}\label{ex_rank_deg}
  In general $\ZR^s(\Q^n)=\ZR^s(\Z^n)$ is not a closed subset for the Patch topology. For instance let us consider the sequence of rank one orders $(\preceq_n)_{n\in\N^*}$ in $\ZR(\Z^2)$ defined by
  $$\preceq_n=\leq_{u_n}$$ 
  where $u_n=(1,\frac{1}{\sqrt{2}n})$.\\
  Let $\preceq=\leq_{(u,v)}$ where $u=(1,0)$ and $v=(0,1)$.\\
  If we consider the filtration of $\Z^2$ given by $\{\mathcal G_k\}_{k\in\N}$ where $\mathcal G_k$ is the set of vectors whose coordinates are in $\{-k,\ldots, k\}$, we see that $\preceq_n$ and $\preceq$ agrees on $\mathcal G_n$. Therefore $(\preceq_n)_n$ converges to $\preceq$ in the Patch topology. But $\rk(\preceq)=2$.
  \end{ex}
  
  \begin{ex}\label{ex_rank_deg2}
Let $(u_n)\in(\Q^2)^\N$ be a sequence of vectors of norm 1 that converges to a vector $u$ with an irrational slope, and let $v_n\in(\Q^2)^\N$ be a sequence of non zero vectors with $v_n\cdot u_n=0$. Then, as in Example \ref{ex_deg}, the sequence of orders $\preceq_n:=\leq_{(u_n,v_n)}$ converges to $\leq_u$. But we have
$$\forall n\ \rk(\preceq_n)=2\text{ and } \rk(\leq_u)=1.$$
This shows that $\ZR^{\leq s}(\Z^n)$ is not closed. This also shows (along with the previous example), that there is no relation between $\limsup \rk(\preceq_n)$ or $\liminf \rk(\preceq_n)$, and $\rk(\lim_n\preceq_n)$.

\end{ex}
  
  \begin{ex}
  Example \ref{ex_deg} shows that $^d\!\ZR(\Q^n)$ is not closed in general. \\
  Again  Example \ref{ex_deg} shows that $^d\!\ZR(\Q^n)$ is not open neither. Indeed for every $n$ we have  $\deg(\preceq_n)=1$, but $\deg(\lim \preceq_n)=0$. Therefore the complement of $^0\ZR(\Z^2)$ is not closed.
  \end{ex}

%

 Hausdorff-Alexandroff Theorem asserts that any compact metric set is the image under a continuous map of a Cantor set. The following result provides an example of such a map in the case of the spheres $\SS^{n-1}$. This generalizes \cite[Proposition 3.1]{S} where such a result is given for $n=2$.
 \begin{prop}\label{HA}
The set of rank one preorders $\ZR^1(\Q^n)$, endowed with the Inverse topology, is homeomorphic to the euclidean sphere $\SS^{n-1}$.\\
 Moreover the map $\pi : \Ord(\Q^n)\lgw \ZR^1(\Q^n)$, where $\pi(\preceq)$ is the unique preorder of rank one such that $\pi(\preceq)\LEQ\preceq$ for every $\preceq\in\Ord(\Q^n)$,   is a continuous surjective map (for the Inverse topology) between an ultrametric Cantor set and the $(n-1)$-dimensional sphere.
\end{prop}

\begin{proof}
The set $\ZR^1(\Q^n)$ is the set of preorders of the form $\leq_u$ for some non zero $u\in\R^n$. In fact we can choose $u$ to be of norm 1, hence $\ZR^1(\Q^n)$ is in bijection with $\SS^{n-1}$. The Inverse topology on $\ZR^1(\Q^n)$ is generated by the $\U_v$ where $v$ runs over the vectors in $\Q^n$. But the bijection between $\ZR^1(\Q^n)$ and $\SS^{n-1}$ induces a bijection between $\U_v$ and the open half sphere $\{u\in\SS^{n-1}\mid u\cdot v>0\}$. Since $\Q$ is dense in $\R$, the sets $\{u\in\SS^{n-1}\mid u\cdot v>0\}$ where $v$ runs over $\Q^n$, generate the euclidean topology. Therefore, $\ZR^1(\Q^n)$ is homeomorphic to the $(n-1)$-dimensional sphere. \\
In order to prove that  $\pi$ is continuous, it is enough to prove that $\pi^{-1}(\U_u)$ is open in $\Ord(\Q^n)$, for every $u\in\Q^n$. Then, let us fix such a $u\in\Q^n$, $u\neq 0$. Let $\preceq\in\pi^{-1}(\U_u)$. Since $\pi(\preceq)\in\U_u$, we have $\preceq=\leq_{(u_1,\ldots, u_s)}$ where $u\cdot u_1>0$. Let $v_1$, \ldots, $v_n\in\Q^n$ be a basis of $\Q^n$, such that $u_1\cdot (u\pm v_i)>0$ for every $i$. Then $\preceq\in \bigcap_{i=1}^n\U_{u+v_i}\cap\U_{u-v_i}$. Moreover, we have
$$\bigcap_{i=1}^n\U_{u+v_i}\cap\U_{u-v_i}\subset \pi^{-1}(\U_u).$$
Indeed, let $\preceq'\in \bigcap_{i=1}^n\U_{u+v_i}\cap\U_{u-v_i}$, and write $\preceq'=\leq_{(u'_1,\ldots, u'_s)}$.
Then $u'_1\cdot (u\pm v_i)\geq 0$ for every $i$. Then $u'_1\cdot u\geq 0$. If $u_1'\cdot u=0$, we have $u'_1\cdot v_i=0$ for every $i$. This is not possible, because $u'_1\neq 1$ and the $v_i$ form a $\Q$-basis of $\Q^n$. Therefore $u'_1\cdot u>0$ and $\preceq'\in\pi^{-1}(\U_u)$. This shows that $\pi^{-1}(\U_u)$ is open in $\Ord(\Q^n)$.\\
Finally  $\Ord(\Q^n)$ is an ultrametric Cantor set by Theorem \ref{main1}.
\end{proof}

 Now we can represent $\ZR(\Q^n)$ as a tree by Proposition \ref{tree} and  Corollary \ref{tree2}. Every preorder corresponds to a vertex of the graph. For a preorder $\preceq\neq \leq_\emptyset$, we consider the largest preorder $\preceq'$ such that $\preceq'\SEQ\preceq$. Every such a pair $(\preceq,\preceq')$ corresponds to an edge between $\preceq$ and $\preceq'$. Moreover $\ZR(\Q^n)$ is a rooted tree by designing $\leq_\emptyset$ to be the root. 
 \begin{ex}

For $n=1$, $\ZR(\Q)$ consists of three elements: the trivial preorder $\leq_\emptyset$ for which $u\leq_\emptyset v$ for every $u$, $v\in{\R_\geq 0}$, and the orders $\leq_1$ and $\leq_{-1}$. Since $\leq_1$ and $\leq_{-1}$ are the two refinements of $\leq_\emptyset$, $\ZR(\Q)$ is a rooted tree with two vertices:

\begin{figure}[H]\fbox{\begin{tikzpicture}[scale=6.9]
\draw[line width=0.5pt] (0,0) -- (1,0);
\draw[line width=0.5pt] (0,0) -- (-1,0);

\filldraw
(0,0) circle (0.5pt) node[align=left,   below] {$\leq_\emptyset$}
 (1,0) circle (0.5pt) node[align=left,   below] {$\leq_1$}
 (-1,0) circle (0.5pt) node[align=right,   below] {$\leq_{-1}$}; 
 
    \end{tikzpicture}}
\caption{The tree $\ZR(\Q)$}
    \end{figure}

\end{ex}
\FloatBarrier
\begin{ex}
For $n=2$, $\ZR(\Q^2)$ can be described as follows:\\
Every order $\preceq$ on ${\R_{\geq 0}}^n$ has the form $\leq_{u_1,u_2}$ where $u_1$ and $u_2$ are nonzero orthonormal vectors. Since $\preceq$ is a preorder on ${\R_{\geq 0}}^n$ we have that $u_1$ is in the dual of ${\R_{\geq 0}}^n$, so $u_1\in {\R_{\geq 0}}^n$. Now if $u_1=\begin{psmallmatrix} a \\  b \end{psmallmatrix}$ has $\Q$-linearly independent coordinates, $\leq_{u_1}$ is already an order and has no refinement, and the data of $u_2$ is superfluous. If the coordinates of $u_1$ are linearly dependent on $\Q$ then we can choose freely $u_2$ in $\langle u_1\rangle ^\perp$. Since $\| u_2\|=1$ there are two possible choices:\\

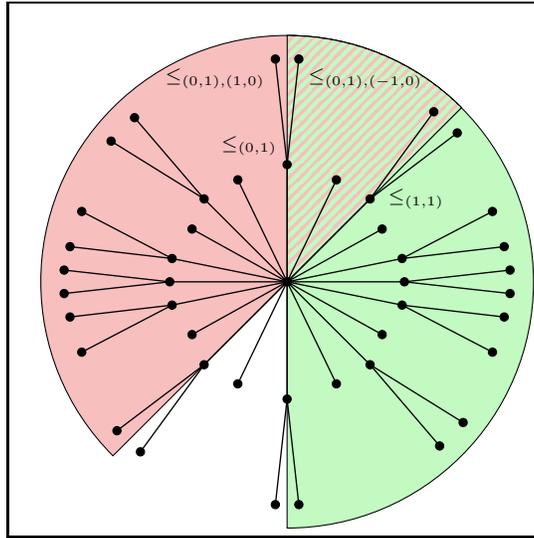
\begin{figure}[H]\fbox{\begin{tikzpicture}[scale=1.56]

   \draw[fill=lightgreen] (0:0) -- (0,2.1) 
  arc[radius = -2.1, start angle= 270, end angle= 90] -- (0:0cm) -- cycle ;
  
  \draw[fill=lightcoral] (0,0) -- (1.486,1.486) 
arc[radius = 2.1, start angle= 45, end angle= 225] -- (0:0cm) -- cycle ;

\draw[pattern=custom north east lines,hatchspread=4pt,hatchthickness=1.6pt,hatchcolor=lightgreen] (0,0) -- (1.486,1.486) 
arc[radius = 2.1, start angle= 45, end angle= 90] -- (0:0cm) -- cycle ;

\draw[line width=0.5pt] (0,0) -- (1,0);

\draw[line width=0.5pt] (0,0) -- (-1,0) ;

\draw[line width=0.5pt] (0,0) -- (0,1) node[above left] {$\scriptstyle\leq_{(0,1)}$} ;

\draw[line width=0.5pt] (0,0) -- (0.81,0.45) ;
\draw[line width=0.5pt] (0,0) -- (-0.81,0.45) ;
\draw[line width=0.5pt] (0,0) -- (-0.81,-0.45) ;

\draw[line width=0.5pt] (0,0) -- (0.42,0.87) ;

\draw[line width=0.5pt] (0,0) -- (-0.42,0.87) ;
\draw[line width=0.5pt] (0,0) -- (-0.42,-0.87) ;

\draw[line width=0.5pt] (0,0) -- (-0.707107,0.707107) ;

\draw[line width=0.5pt] (-1,0) -- (-1.9,0.1) ;
\draw[line width=0.5pt] (1,0) -- (1.9,0.1) ;


\draw[line width=0.5pt] (0,1) -- (0.1,1.9)  node[below  right] {$\scriptstyle\leq_{(0,1), (-1,0)}$} ;
\draw[line width=0.5pt] (0,1) -- (-0.1,1.9) node[below left] {$\scriptstyle\leq_{(0,1), (1,0)}$}   ;

\draw[line width=0.5pt] (0,0) -- (0.707107,0.707107) node[ right] {$\ \scriptstyle\leq_{(1,1)}$};
\draw[line width=0.5pt] (0,0) -- (0.98,0.1986693308) ;

\draw[line width=0.5pt] (0,0) -- (-0.98,0.1986693308) ;
\draw[line width=0.5pt] (0,0) -- (-0.98,-0.1986693308) ;

\draw[line width=0.5pt] (0.707107,0.707107) -- (1.45 ,1.27 ) ; 
\draw[line width=0.5pt] (0.707107,0.707107) -- (1.25, 1.45);

\draw[line width=0.5pt] (-0.707107,-0.707107) -- (-1.45 ,-1.27 ) ; 
\draw[line width=0.5pt] (-0.707107,-0.707107) -- (-1.25, -1.45);

\draw[line width=0.5pt] (0.98,0.1986693308) -- (1.85, 0.3) ;
\draw[line width=0.5pt] (0.98,0.1986693308) -- (1.75, 0.6);

\draw[line width=0.5pt] (0,0) -- (0,-1) ;

\draw[line width=0.5pt] (0,0) -- (0.81,-0.45) ;

\draw[line width=0.5pt] (0,0) -- (0.42,-0.87) ;

\draw[line width=0.5pt] (0,0) -- (-0.707107,-0.707107) ;

\draw[line width=0.5pt] (-1,0) -- (-1.9,-0.1) ;
\draw[line width=0.5pt] (1,0) -- (1.9,-0.1) ;

\draw[line width=0.5pt] (0,-1) -- (0.1,-1.9) ;
\draw[line width=0.5pt] (0,-1) -- (-0.1,-1.9) ;

\draw[line width=0.5pt] (0,0) -- (0.707107,-0.707107) ;
\draw[line width=0.5pt] (0,0) -- (0.98,-0.1986693308) ;

\draw[line width=0.5pt] (0.707107,-0.707107) -- (1.3, -1.4) ; 
\draw[line width=0.5pt] (0.707107,-0.707107) -- (1.5 ,-1.2);

\draw[line width=0.5pt] (-0.707107,0.707107) -- (-1.5 ,1.2 ) ; 
\draw[line width=0.5pt] (-0.707107,0.707107) -- (-1.3, 1.4);

\draw[line width=0.5pt] (0.98,-0.1986693308) -- (1.85, -0.3) ;
\draw[line width=0.5pt] (0.98,-0.1986693308) -- (1.75, -0.6);

\draw[line width=0.5pt] (-0.98,0.1986693308) -- (-1.85, 0.3) ;
\draw[line width=0.5pt] (-0.98,0.1986693308) -- (-1.75, 0.6);

\draw[line width=0.5pt] (-0.98,-0.1986693308) -- (-1.85, -0.3) ;
\draw[line width=0.5pt] (-0.98,-0.1986693308) -- (-1.75, -0.6);
   
\filldraw (0,0) circle (1.pt) 
 (-1,0) circle (1.pt) 
 (1,0) circle (1.pt) 
  (-1.9,0.1) circle (1.pt) 
   (1.9,0.1) circle (1.pt)
   (0.42,0.87) circle (1.pt)
    (0,1) circle (1.pt)
    (0.81,0.45) circle (1.pt)
     (-0.707107,0.707107) circle (1.pt) 

   (1.45 ,1.27 ) circle (1.pt)
         (-1.5 ,1.2 ) circle (1.pt) 
        
        (0.707107,0.707107) circle (1.pt)
          (0.98,0.1986693308) circle (1.pt)
                    (-0.98,0.1986693308) circle (1.pt)
                         (-0.98,-0.1986693308) circle (1.pt)
       (1.25, 1.45) circle (1.pt)
              (-1.3, 1.4) circle (1.pt)
    (1.85, 0.3) circle (1.pt)
     (-1.85, 0.3) circle (1.pt)   
      (-1.85, -0.3) circle (1.pt)         
    (1.75, 0.6) circle (1.pt) 
     (-1.75, 0.6) circle (1.pt)     
       (-1.75, -0.6) circle (1.pt)     
    (0.1,1.9) circle (1.pt)  
     (-0.1,1.9) circle (1.pt)         
    (-1.9,-0.1) circle (1.pt) 
   (1.9,-0.1) circle (1.pt)
   (0.42,-0.87) circle (1.pt)
      (-0.42,0.87) circle (1.pt)
      (-0.42,-0.87) circle (1.pt)
    (0,-1) circle (1.pt)
    (0.81,-0.45) circle (1.pt)
      (-0.81,-0.45) circle (1.pt)
        (-0.81,0.45) circle (1.pt)
     (-0.707107,-0.707107) circle (1.pt) 
      (1.5 ,-1.2 ) circle (1.pt)
         (-1.45 ,-1.27 ) circle (1.pt) 
        (0.707107,-0.707107) circle (1.pt)
          (0.98,-0.1986693308) circle (1.pt)
            (1.3, -1.4) circle (1.pt)
             (-1.25, -1.45)circle (1.pt) 
    (1.85, -0.3) circle (1.pt)         
    (1.75, -0.6) circle (1.pt)     
    (0.1,-1.9) circle (1.pt)  
     (-0.1,-1.9) circle (1.pt)         
;
 \end{tikzpicture}}

\caption{The tree $\ZR(\Q^2)$. The green zone is the set $\O_{(1,0)}$ (if we include the boundary), or $\U_{(1,0)}$ (if we remove the boundary). The red zone is $\O_{(-1,1)}$, and the intersection of both is $\O_{(1,0)}\cap\O_{(-1,1)}$.}
\end{figure}
\end{ex} 

\FloatBarrier

\begin{ex}
In dimension 3, we have the following picture:
\begin{figure}[H]
\fbox{\begin{tikzpicture}[scale=0.87]

\draw[line width=0.5pt] (0,0) -- (2,0);

\draw[line width=0.5pt] (0,0) -- (-1.6,-.7);
\draw[line width=0.5pt] (0,0) -- (-1.69,-1.1);
\draw[line width=0.5pt] (0,0) -- (-0.6,1.5);
\draw[line width=0.5pt] (0,0) -- (0.5,-1.7);
\draw[line width=0.5pt] (0,0) -- (1.92,0.55);
\draw[line width=0.5pt] (0,0) -- (1.79,0.9);
\draw[line width=0.5pt] (0,0) -- (.2,1.9);
\draw[line width=0.5pt] (0,0) -- (-1.7,0.55);
\draw[line width=0.5pt] (0,0) -- (-0.9,-.2);
\draw[line width=0.5pt] (0,0) -- (-.3,-1.2);
\draw[line width=0.5pt] (0.5,-1.7)-- (0.65,-2.5);
\draw[line width=0.5pt] (0.5,-1.7)-- (0.85,-2.45);
\draw[line width=0.5pt] (-1.69,-1.1) -- (-2.5,-1.3);
\draw[line width=0.5pt] (-1.69,-1.1) -- (-2.15,-1.8);
\draw[line width=0.5pt] (1.4,-0.43) -- (2.3,-0.5);
\draw[line width=0.5pt] (1.4,-0.43) -- (2.2,-0.85);
\draw[line width=0.5pt] (-0.6,1.5) --(-0.75,2.3);
\draw[line width=0.5pt] (-0.6,1.5) -- (-1.15,2.15);

\draw[line width=0.5pt] (0,0) -- (-2,0) ;

\draw[line width=0.5pt] (0,0) -- (0,2) ;

\draw[line width=0.5pt] (0,0) -- (1.4,-0.43);

\draw[line width=0.5pt] (0,0) -- (-0.6,-0.57);
\draw[line width=0.5pt] (0,0) -- (0.6,0.57);
\draw[line width=0.5pt] (0,0) -- (-0.6,0.57);

\draw[line width=0.5pt] (0,0) -- (0,-2) ;

\draw[line width=0.5pt] (0,0) -- (0.7,1.3);

\draw[line width=0.5pt] (2,0) -- (2.77,0.3) ;

\draw[line width=0.5pt] (2,0) -- (3,1) ;
\draw[line width=0.5pt] (2,0) -- (3,-1) ;
\draw[line width=0.5pt] (2,0) -- (3.24,0.2) ;
\draw[line width=0.5pt] (2,0) -- (3.23,-0.4) ;
\draw[line width=0.5pt] (2,0) -- (2.81,-0.6) ;
\draw[line width=0.5pt] (2,0) -- (3.25,-0.15) ;
\draw[line width=0.5pt] (2,0) -- (3.18,0.7) ;

\draw[line width=0.5pt] (3.25,-0.15) -- (4,-0.25) ;
\draw[line width=0.5pt] (3.25,-0.15) -- (4,-0.05) ;

\draw[line width=0.5pt] (2.81,-0.6) -- (3.6,-0.5) ;
\draw[line width=0.5pt] (2.81,-0.6) -- (3.6,-0.7) ;

\draw[line width=0.5pt] (3,1) -- (3.8,1.10) ;
\draw[line width=0.5pt] (3,1) -- (3.8,.90) ;

\draw[line width=0.5pt] (3,-1) -- (3.8,-1.10) ;
\draw[line width=0.5pt] (3,-1) -- (3.8,-.90) ;

\draw[line width=0.5pt] (-2,0) -- (-2.77,0.3) ;

\draw[line width=0.5pt] (-2,0) -- (-3,1) ;
\draw[line width=0.5pt] (-2,0) -- (-3,-1) ;
\draw[line width=0.5pt] (-2,0) -- (-3.24,0.2) ;
\draw[line width=0.5pt] (-2,0) -- (-3.23,-0.4) ;
\draw[line width=0.5pt] (-2,0) -- (-2.81,-0.6) ;
\draw[line width=0.5pt] (-2,0) -- (-3.18,0.7) ;
\draw[line width=0.5pt] (-2,0) -- (-3.25,-0.15) ;

\draw[line width=0.5pt] (-3.25,-0.15) -- (-4,-0.25) ;
\draw[line width=0.5pt] (-3.25,-0.15) -- (-4,-0.05) ;

\draw[line width=0.5pt] (-2.81,-0.6) -- (-3.6,-0.5) ;
\draw[line width=0.5pt] (-2.81,-0.6) -- (-3.6,-0.7) ;

\draw[line width=0.5pt] (-3,1) -- (-3.8,1.10) ;
\draw[line width=0.5pt] (-3,1) -- (-3.8,.90) ;

\draw[line width=0.5pt] (-3,-1) -- (-3.8,-1.10) ;
\draw[line width=0.5pt] (-3,-1) -- (-3.8,-.90) ;
\draw[line width=0.5pt] (0,-2) -- (0.3,-2.77) ;

\draw[line width=0.5pt] (0,-2) -- (1,-3) ;
\draw[line width=0.5pt] (0,-2) -- (-1,-3) ;
\draw[line width=0.5pt] (0,-2) -- (0.2,-3.24) ;
\draw[line width=0.5pt] (0,-2) -- (-0.4,-3.23) ;
\draw[line width=0.5pt] (0,-2) -- (-0.6,-2.81) ;
\draw[line width=0.5pt] (0,-2) -- (-0.15,-3.25) ;
\draw[line width=0.5pt] (0,-2) -- (0.7,-3.18) ;

\draw[line width=0.5pt] (-0.15,-3.25) -- (-0.25,-4) ;
\draw[line width=0.5pt] (-0.15,-3.25) -- (-0.05,-4) ;

\draw[line width=0.5pt] (-0.6,-2.81) -- (-0.5,-3.6) ;
\draw[line width=0.5pt] (-0.6,-2.81) -- (-0.7,-3.6) ;

\draw[line width=0.5pt] (1,-3) -- (1.10,-3.8) ;
\draw[line width=0.5pt] (1,-3) -- (.90,-3.8) ;

\draw[line width=0.5pt] (-1,-3) -- (-1.10,-3.8) ;
\draw[line width=0.5pt] (-1,-3) -- (-.90,-3.8) ;
\draw[dashed, ultra thin] (0,3) circle [x radius=10mm, y radius=0.25cm];
\draw[dashed, ultra thin] (0,-3) circle [x radius=10mm, y radius=0.25cm];
\draw[line width=0.5pt] (0,2) -- (0.3,2.77) ;

\draw[line width=0.5pt] (0,2) -- (1,3) ;
\draw[line width=0.5pt] (0,2) -- (-1,3) ;
\draw[line width=0.5pt] (0,2) -- (0.2,3.24) ;
\draw[line width=0.5pt] (0,2) -- (-0.4,3.23) ;
\draw[line width=0.5pt] (0,2) -- (-0.6,2.81) ;
\draw[line width=0.5pt] (0,2) -- (-0.15,3.25) ;
\draw[line width=0.5pt] (0,2) -- (0.7,3.18) ;

\draw[line width=0.5pt] (-0.15,3.25) -- (-0.25,4) ;
\draw[line width=0.5pt] (-0.15,3.25) -- (-0.05,4) ;

\draw[line width=0.5pt] (-0.6,2.81) -- (-0.5,3.6) ;
\draw[line width=0.5pt] (-0.6,2.81) -- (-0.7,3.6) ;

\draw[line width=0.5pt] (1,3) -- (1.10,3.8) ;
\draw[line width=0.5pt] (1,3) -- (.90,3.8) ;

\draw[line width=0.5pt] (-1,3) -- (-1.10,3.8) ;
\draw[line width=0.5pt] (-1,3) -- (-.90,3.8) ;

\draw[line width=0.5pt] (0,0) -- (-1.05,-1.7) ;
\draw[line width=0.5pt] (0,0) -- (-1.6,1.2) ;
\draw[line width=0.5pt] (0,0) -- (1.2,-1.3) ;

\draw[line width=0.5pt] (1.2,-1.3) -- (2,-1.6) ;
\draw[line width=0.5pt] (1.2,-1.3) -- (1.5,-2.1) ;

  \draw[dashed, line width=0.01pt] (0,0) circle (2cm);
  \draw[dashed, line width=0.01pt] (-2,0) arc (180:360:2 and 0.6);
  \draw[dashed,line width=0.01pt] (2,0) arc (0:180:2 and 0.6);
  \fill[fill=black] (0,0) circle (1pt);
\draw[dashed, ultra thin] (3,0) circle [x radius=.25cm, y radius=10mm];

\draw[dashed, ultra thin] (-3,0) circle [x radius=.25cm, y radius=10mm];

\filldraw (2,-1.6) circle (1.pt);
\filldraw (1.5,-2.1) circle (1.pt);
\filldraw (-1.05,-1.7) circle (1.pt);
\filldraw (-1.6,1.2) circle (1.pt);
\filldraw (1.2,-1.3) circle (1.pt);
\filldraw (0,0) circle (1.pt);
\filldraw (2,0) circle (1.pt);
\filldraw (-2,0) circle (1.pt);
\filldraw (3,1) circle (1.pt);
\filldraw (3,-1) circle (1.pt);
\filldraw (2.81,-0.6) circle (1.pt);
\filldraw (3.25,-0.15) circle (1.pt);
\filldraw (3.8,-.90) circle (1.pt);

\filldraw (3.8,-1.10) circle (1.pt);
\filldraw (3.6,-0.7) circle (1.pt);
\filldraw (3.6,-0.5) circle (1.pt);
\filldraw (4,-0.05) circle (1.pt);
\filldraw (4,-0.25) circle (1.pt);
\filldraw (3.18,0.7) circle (1.pt);
\filldraw (3.23,-0.4) circle (1.pt);
\filldraw (3.24,0.2) circle (1.pt);
\filldraw (2.77,0.3) circle (1.pt);

\filldraw (0.65,-2.5) circle (1.pt);
\filldraw (0.85,-2.45) circle (1.pt);
\filldraw (-2.5,-1.3) circle (1.pt);
\filldraw  (-2.15,-1.8) circle (1.pt);
\filldraw  (2.3,-0.5) circle (1.pt);
\filldraw  (2.2,-0.85) circle (1.pt);
\filldraw (-0.75,2.3) circle (1.pt);
\filldraw  (-1.15,2.15) circle (1.pt);
\filldraw (-1.6,-.7) circle (1.pt);
\filldraw (-1.69,-1.1) circle (1.pt);
\filldraw (-0.6,1.5) circle (1.pt);
\filldraw (0.5,-1.7) circle (1.pt);
\filldraw (1.92,0.55) circle (1.pt);
\filldraw (.2,1.9) circle (1.pt);
\filldraw (1.79,0.9) circle (1.pt);
\filldraw (-1.7,0.55) circle (1.pt);
\filldraw (-0.9,-.2) circle (1.pt);
\filldraw (-.3,-1.2) circle (1.pt);
\filldraw (1,3) circle (1.pt);
\filldraw (-1,3) circle (1.pt);
\filldraw (-0.6,2.81) circle (1.pt);
\filldraw (-0.15,3.25) circle (1.pt);
\filldraw (-.90,3.8) circle (1.pt);

\filldraw (-1.10,3.8) circle (1.pt);
\filldraw (-0.7,3.6) circle (1.pt);
\filldraw (-0.5,3.6) circle (1.pt);
\filldraw (-0.05,4) circle (1.pt);
\filldraw (-0.25,4) circle (1.pt);
\filldraw (0.7,3.18) circle (1.pt);
\filldraw (-.4,3.23) circle (1.pt);
\filldraw (.2,3.24) circle (1.pt);
\filldraw (.3,2.77) circle (1.pt);

\filldraw (1.1,3.8) circle (1.pt);
\filldraw (.9,3.8) circle (1.pt);
\filldraw (1,-3) circle (1.pt);
\filldraw (-1,-3) circle (1.pt);
\filldraw (-0.6,-2.81) circle (1.pt);
\filldraw (-0.15,-3.25) circle (1.pt);
\filldraw (-.90,-3.8) circle (1.pt);

\filldraw (-1.10,-3.8) circle (1.pt);
\filldraw (-0.7,-3.6) circle (1.pt);
\filldraw (-0.5,-3.6) circle (1.pt);
\filldraw (-0.05,-4) circle (1.pt);
\filldraw (-0.25,-4) circle (1.pt);
\filldraw (0.7,-3.18) circle (1.pt);
\filldraw (-.4,-3.23) circle (1.pt);
\filldraw (.2,-3.24) circle (1.pt);
\filldraw (.3,-2.77) circle (1.pt);

\filldraw (1.1,-3.8) circle (1.pt);
\filldraw (.9,-3.8) circle (1.pt);
\filldraw (-2.81,-0.6) circle (1.pt);
\filldraw (-3.25,-0.15) circle (1.pt);
\filldraw (-3.8,-.90) circle (1.pt);

\filldraw (-3.8,-1.10) circle (1.pt);
\filldraw (-3.6,-0.7) circle (1.pt);
\filldraw (-3.6,-0.5) circle (1.pt);
\filldraw (-4,-0.05) circle (1.pt);
\filldraw (-4,-0.25) circle (1.pt);
\filldraw (-3.18,0.7) circle (1.pt);
\filldraw (-3.23,-0.4) circle (1.pt);
\filldraw (-3.24,0.2) circle (1.pt);
\filldraw (-2.77,0.3) circle (1.pt);

\filldraw (-3.8,1.10) circle (1.pt);
\filldraw (-3.8,.90) circle (1.pt);

\filldraw (-0.6,-0.57) circle (1.pt);

\filldraw (3.8,1.10) circle (1.pt);
\filldraw (3.8,.90) circle (1.pt);
\filldraw (0,2) circle (1.pt);
\filldraw (0,-2) circle (1.pt);
\filldraw (0.6,0.57) circle (1.pt);
\filldraw (-0.6,0.57) circle (1.pt);
\filldraw (1.4,-0.43) circle (1.pt);
\filldraw (0.7,1.3) circle (1.pt);
\end{tikzpicture}}

\caption{The tree $\ZR(\Q^3)$}.
\end{figure}
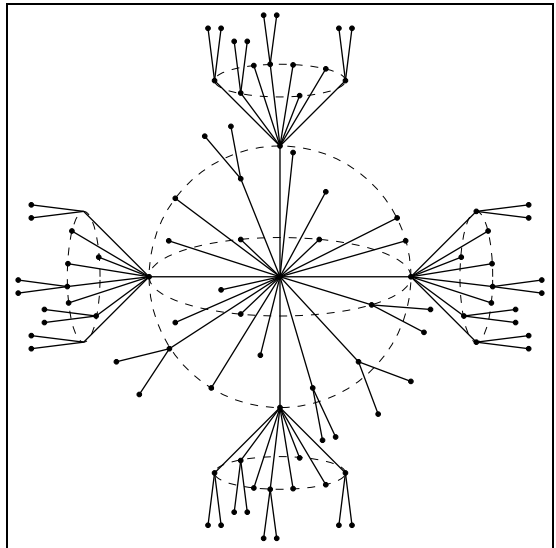

\end{ex}

\subsection{Some non commutative groups}
\subsubsection{The Klein Bottle group}
Let $G=\langle x,y\mid xyx^{-1}=y^{-1}\rangle$. This is the fundamental group of the Klein Bottle.

\begin{lem}\label{lem_conv}
Let $\preceq\in\ZR_l(G)$ with  $x\not\sim_\preceq 1$. Then the subgroup generated by $y$ is a $\preceq$-isolated normal subgroup of $G$.
\end{lem}

\begin{proof}
Clearly $\langle y\rangle$ is a normal subgroup of $G$. Let us prove that it is $\preceq$-isolated. \\
We remark that every element of $G$ can be written as $y^mx^n$, with $m$, $n\in\Z$. If $y\sim_\preceq 1$, then if, for some $k$, $l$, $m$, $n\in\Z$ we have 
$$y^k\preceq y^mx^n\preceq y^l$$
then $x^n\sim_\preceq 1$, and $n=0$ since $x\not\sim_\preceq 1$. Thus in this case, $\langle y\rangle$ is $\preceq$-isolated.
 Therefore we may assume that $y\not\sim_\preceq 1$ and $x\not\sim_\preceq 1$.\\
Assume that $y\succ 1$ and $x\succ 1$. Then, for every $k\in\N$, $xy^k\succ 1$. But $xy^k=y^{-k}x\succ 1$. Therefore $x\succ y^k$ for every $k\in\N$. In the same way, $x^{-1}y^{-k}=y^kx^{-1}\prec 1$, and $x^{-1}\prec y^{-k}$ for every $k\in\N$.\\
Now let $n$, $m\in\N$. If $n>0$ and $m\in\Z$, then  
$$y^mx^n\succ y^k\ \  \forall k.$$
In the same way, for every $k\in\N$, $n\in\N$, $m\in\Z$,
$y^{m}x^{-n}\prec y^{-k}$. This proves that $\langle y\rangle$ is $\preceq$-isolated.\\
Now if $y\prec 1$ and $x\prec 1$, we replace $x$ and $y$ by $x^{-1}$ and $y^{-1}$,  and, since the relation $xyx^{-1}=y^{-1}$ can be rewritten as $x^{-1}y^{-1}x=y$,  the result follows from the previous case. \\If $x\prec 1$ and $y\succ 1$, we remark that every element of $G$ can be written as $x^my^n$, with $m$, $n\in\Z$. Since $y^{-k}x^{-1}=x^{-1}y^{k}\succ 1$, we have $x^{-1}\succ y^k$ for every $k\in\N$. In the same way $x\prec y^{-k}$ for every $k\in\N$. Therefore the same reasoning applies. The case $y\succ 1$ and $x\prec 1$ is obtained by replacing $x$ and $y$ by $x^{-1}$ and $y^{-1}$.
\end{proof}
Therefore for every preorder $\preceq$, we have $G_\preceq=\langle 1\rangle$, $\langle y\rangle$, or $G$. \\
If $G_\preceq=G$, $\preceq=\preceq_\emptyset$ is bi-invariant. \\
Now, let $\preceq\in\ZR_l(G)$ and assume $x\sim_\preceq 1$. Since $xy=y^{-1}x$, we have $y^{-1}\preceq xy\preceq y^{-1}$. Therefore $y\sim_\preceq 1$, and $\preceq_\emptyset$ is the only preorder for which $x$ is equivalent to $1$.\\
 If $G_\preceq=\langle y\rangle$, $\preceq$ is the composition of an order on $G/\langle y\rangle\simeq \Z$ with the trivial preorder on $\langle y\rangle$, which is completely determined by the sign of $x$. We denote by $\preceq_{+1}$ (resp. $\preceq_{-1}$) the preorder such that $x$ is positive (resp. negative). Therefore there are two such preorders, and these are bi-invariant.\\
 Finally, if $G_\preceq=\langle 1\rangle$, since $\langle y\rangle$ is $\preceq$-isolated, the order $\preceq$ is lexicographically defined by the following short exact sequence  (see  Proposition \ref{lex_order2}):
\begin{equation}\label{short_klein}1\lgw \langle y\rangle\lgw G\lgw G/\langle y\rangle\simeq \Z\lgw 1\end{equation}  
But, since $\langle y\rangle\simeq\Z$, the only possible orders are determined by their (positive or negative) signs on $x$ and $y$. We denote by $\preceq_{\e_1,\e_2}$ the order for which the sign of $x$ (resp. of $y$) is $\e_1$ (resp. $\e_2$), where $\e_i=\pm 1$. Moreover we have
 $$\preceq_{+1}\SEQ \preceq_{+1,\e_2}\ \ \forall \e_2,$$
 $$\preceq_{-1}\SEQ \preceq_{-1,\e_2}\ \ \forall \e_2.$$ 
Finally, these orders are not bi-invariant since $xyx^{-1}=y^{-1}$.

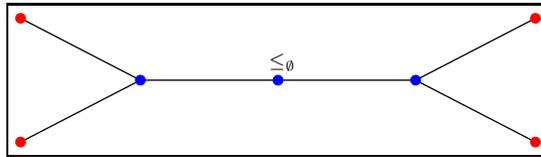
\begin{figure}[H]\fbox{\begin{tikzpicture}[scale=1.83]

\draw[line width=0.5pt] (0,0) node[ above] {$\ \scriptstyle\leq_{\emptyset}$} -- (1,0);

\draw[line width=0.5pt] (0,0) -- (-1,0) ;


\draw[line width=0.5pt] (1,0) -- (1.87,0.45) ;
\draw[line width=0.5pt] (1,0) -- (1.87,-0.45) ;
\draw[line width=0.5pt] (-1,0) -- (-1.87,-0.45) ;
\draw[line width=0.5pt] (-1,0) -- (-1.87,+0.45) ;

\filldraw[blue] (0,0) circle (1.pt) ;
\filldraw[blue] (1,0) circle (1.pt);
\filldraw[blue] (-1,0) circle (1.pt);
\%filldraw[blue] (0,-0.6) circle (1.pt);

\filldraw[red] (1.87,0.45) circle (1.pt);
\filldraw[red]  (-1.87,0.45) circle (1.pt);
\filldraw[red]  (1.87,-0.45) circle (1.pt);
\filldraw[red]  (-1.87,-0.45) circle (1.pt);
\end{tikzpicture}}
\caption{The tree $\ZR_l(G)$ where $G$ is the Klein Bottle group. The bi-invariant preorders are in blue, the other ones in red.}.
\end{figure}

\subsubsection{Some groups with no non-trivial preorders}
If $G$ is a torsion group, we have that $\ZR_l(G)=\{\leq_{\emptyset}\}$. But there are also torsion free groups for which $\ZR_l(G)=\{\leq_\emptyset\}$. One example is given as follows (this example is the group $G'$ of \cite[Example 1.64]{CR} for which it is shown that $\Ord_l(G')=\emptyset$).\\
We consider the Klein bottle group $G$ given in the previous example. If we set
$$a=xy \text{ and } b=y^{-1}xy$$
we obtain the presentation
$G=\lg a,b\mid a^2=b^2\rg.$
We remark that the subgroup $H$ generated by $a^2$ and $ab$, is isomorphic to $\Z^2$.\\
Now we consider the group given in \cite[Example 1.64]{CR}: we consider two copies of $G$, denoted by $G_1$, $G_2$, whose generators are $a_1$, $b_1$ and $a_2$, $b_2$ respectively. We denote by $H_1$ and $H_2$ the respective subgroups isomorphic to $H$. We denote by $G'$ the amalgamated free product $G_1*G_2$ along $H_1\simeq H_2$, where the isomorphism between $H_1$ and $H_2$ is given by
$$a_1^2=(a_2^2)^p(a_2b_2)^q\text{ and } a_1b_1=(a_2)^r(a_2b_2)^s$$
where $p$, $q$, $r$, $s\in\Z$ and $ps-rq=\pm 1$.
We have the presentation
$$G'=\lg a_1,a_2,b_1,b_2\mid a_1^2=b_1^2, a_2^2=b_2^2, a_1^2=(a_2^2)^p(a_2b_2)^q\text{ and } a_1b_1=(a_2)^r(a_2b_2)^s\rg.$$
Now assume that $p$, $q\geq 0$ and $r$, $s\leq 0$ and let $\preceq\in\ZR_l(G')$. Then the first relation in $G'$ implies that $a_1$ and $b_1$ are both non-negative or both non-positive for $\preceq$. In the same way $a_2$ and $b_2$ are both non-negative or both non-positive for $\preceq$. \\
The third relation implies that if $a_2$, $b_2\succeq 1$ (resp. $a_2$, $b_2\preceq 1$), then $a_1\succeq 1$ (resp $a_1\preceq 1$). But the last relation implies that
$$a_2, b_2\succeq 1 \ (\text{resp. } a_2,b_2\preceq 1) \Longrightarrow a_1b_1\preceq 1 \ (\text{resp. } a_1b_1\succeq 1)$$
Therefore we have $a_1$, $b_1\sim_\preceq 1$. This implies that $a_2$, $b_2\sim_\preceq 1$. Therefore $\preceq=\leq_{\emptyset}$.\\
Moreover $G'$ is torsion free, since it is the amalgamated product of two torsion free groups.\\
\\
Now, if we set $G''=G'\times \Z$, $G''$ is a torsion free group, $G'$ is a normal subgroup of $G''$, and $G'\subset G''_\preceq$ for every $\preceq\in\ZR_l(G'\times \Z)$ by the previous reasoning. Therefore $\ZR_l(G'\times\Z)$ is homeomorphic to $\ZR_l(\Z)=\ZR(\Z)$. Thus $\Ord_l(G'')=\emptyset$, and $\ZR_l(G'')=\ZR(G'')\neq \{\leq_\emptyset\}$. 


\section{The Zariski-Riemann space of valuations}
\subsection{From preorders to (monomial) valuations}
\begin{defi}
A pair $(G,\preceq)$ is called a \emph{$l$-group} if $G$ is a group, $\preceq\in\Ord(G)$ (in particular it is bi-invariant), and $G$ is lattice with respect to the order $\preceq$.
\end{defi}

\begin{defi}\cite{Sc}\label{def_val}
Let $\K$ be a division ring and $(G,\preceq)$ be a $l$-group. A \emph{valuation} on $\K$ with values in $G$ is a surjective function $\nu:\K\lgw G\cup\{\infty\}$ such that
\begin{enumerate}
\item[i)] $\nu(0)=\infty\succ u$ for all $u\in G$, and $\nu^{-1}(\infty)=\{0\}$,
\item[ii)] $\nu(uv)=\nu(u)+\nu(v)$ for all $u$, $v\in G$,
\item[iii)] $\nu(u+v)\succeq\min\{\nu(u),\nu(v)\}$ for all $u$, $v\in G$.
\end{enumerate}
In this case, the group $G$ is the \emph{value group} of $\nu$ and is denoted by $\G_\nu$.
\end{defi}

\begin{rem}
The set $V_\nu:=\{x\in\K\mid \nu(x)\geq 0\}$ is a ring with the two following properties:
\begin{enumerate}
\item[a)] $\forall a\in V_\nu$, $\forall b\in\K^*$, $bab^{-1}\in V_\nu$.
\item[b)] $\forall a\in\K^*$ , $a\in V_\nu$ or $a^{-1}\in V_\nu$.
\end{enumerate}
On the other hand, every subring $V\subset\K$ satisfying a) and b) is called a \emph{valuation ring} and there is a $l$-group $G$ and a valuation $\nu:\K\lgw G\cup\{\infty\}$ such that $V=V_\nu$. See \cite{Sc} for the details.
\end{rem}

\begin{defi}
Let $G$ be a  group and $\k$ be a field.
Let us denote by $\K_G^\k$ the division $\k$-algebra of non commutative polynomials with exponents in $G$ and coefficients in $\k$:
$$\K_G^\k:=\left\{\sum_{g\in G}a_gx^g\mid a_g\in \k\right\}$$
where the addition is defined term by term and the multiplication is defined by $x^g x^{g'}:=x^{gg'}$.  It has a multiplicative identity which is $x^1$.
\end{defi}

\begin{defi}\label{val}
Let $G$ be a group and let $\preceq\in \ZR(G)$ such that $(G/G_\preceq,\preceq)$ is a $l$-group. Then $\preceq$ defines a monomial valuation $\nu_\preceq$ on $\K_G^\k$ in the following way:\\
$$\nu_\preceq(x^g)=\ovl g\in G/G_{\preceq}$$
where $\ovl g$ denotes the equivalence class of $g$ under $\sim_\preceq$.
\end{defi}

\begin{defi}
Let $V$ be a valuation ring and let $\p$ be a two-sided prime ideal. Then the localization $V_\p$ is the set of equivalence classes $(v,s)\in V\times (V\backslash\p)$ under the equivalence
$$(v,s)\sim (v',s') \text{ if } vs'=v's.$$
the set of equivalent classes is denoted by $\frac{v}{s}$. This a ring because for every $(v,s)\in V\times (V\backslash\p)$ there exist $(v',s')\in V\times (V\backslash\p)$ such that
$$vs'=sv'.$$
This comes from the fact that $\p$ is a two-sided ideal (see \cite[Lemma 7]{Sc} for the details).
\end{defi}

\begin{defi}\cite[Theorem 2]{Sc}\cite[Proposition 4.1]{Va}
Let $G$ be a group and $H$ be an isolated normal proper subgroup of $G$. Let $V$ be a valuation ring of value group $G$ and let  $\nu:V\lgw G$ be the associated valuation. Let $\p$ be the two-sided prime ideal of $V$ corresponding to $H$, that is $\p=\nu^{-1}(G\backslash H)$. Then $V_\p$ is a valuation ring of value group $G/H$, whose valuation is denoted by $\nu'$ and is defined by 
$$\nu'(v/s)=\nu( v)\text{ mod. }H.$$ 
On the other hand $V/\p$ is a valuation ring of value group $H$, whose valuation is denoted by $\ovl\nu$, and is defined by
$$\forall a\notin\p,\ \ovl\nu(a+\p)=\nu(a).$$
We say that $\nu$ is the composition of $\nu'$ and $\ovl\nu$ and denote
$$\nu=\nu'\circ\ovl\nu.$$

\end{defi}
The following proposition shows that the composition of preorders corresponds to the composition of valuations:

\begin{prop}\label{ca se compose}
Let $G$ be a group 	and $\k$ be a field. Let $\preceq_1\in \Ord(G)$ and $\preceq_2\in \ZR(G)$  with $\preceq_2\,\boldsymbol{\leq}\,\preceq_1$, and such that $G_{\preceq_2}$ is a normal subgroup of $G$. We have three valuations
$$\nu_{\preceq_1}: \K_G^\k\lgw G,$$
$$\nu_{\preceq_2}:\K_{G}^\k\lgw G/G_{\preceq_2},$$
$$\nu_{\preceq_3}:\K_{G_{\preceq_2}}^\k\lgw G_{\preceq_2}$$
where $\preceq_3$ is the restriction of $\preceq_1$ to $G_{\preceq_2}$, that is $\preceq_1=\preceq_2\circ\preceq_3$. 
Then 
$$\nu_{\preceq_1}=\nu_{\preceq_2}\circ\nu_{\preceq_3}.$$
\end{prop}

\begin{proof}
Let $V$ be the valuation ring of $\nu_1$. The set 
$$\p:=\{f\in V\mid \nu_{\preceq_1}(f)\in G\backslash G_{\preceq_2}\}$$
is a two-sided prime ideal of $V$. The ring $V_\p$ is a valuation ring with value group $G/G_{\preceq_2}$. Its valuation $\nu'$ is the valuation sending an element $\frac{v}{s}  $, for $s\in V\backslash\p$ and $v\in V$, onto the class $\ovl v$  of $v\in G$ in $G/G_{\preceq_2}$. Thus $\nu'_{|\K_G^\k}=\nu_{\preceq_2}$.\\
Now the ring $V/\p$ is a valuation ring with value group $G_{\preceq_2}$. We denote by $\ovl\nu$ its valuation. For an element $v\in V\backslash \p$, $\ovl\nu(v+\p)=\nu_{\preceq_1}(v)$. By construction we have
$$\K_{G_{\preceq_2}}^\k=\K_G^\k/(\p\cap \K_G^\k).$$
Hence, by Definition \ref{val} we have that $\ovl\nu=\nu_{\preceq_3}$.
\end{proof}


\subsection{The Zariski-Riemann space}
From now on we will consider the commutative case, that is, we only consider valuations defined on a field.
\begin{defi}\label{ZR0}
Let $\K$ be a field. Let $\nu$ be a valuation on $\K$, that is, a function $\nu:\K\lgw G\cup\{\infty\}$, where $G$ is a totally ordered abelian group, satisfying Definition \ref{def_val}. For such a valuation, we denote by $V_\nu$ its valuation ring, and by $\m_\nu$ its maximal ideal.\\
We define an equivalence relation on the set of such valuations: two valuations $\nu$ and $\mu$ are equivalent if $V_\nu=V_\mu$ or, equivalently, if there is a non zero constant $\la$ such that $\nu(x)=\la\mu(x)$ for every $x\in \K$. \\
The set of such valuations modulo this equivalence relation is called the Zariski-Riemann space of $\K$ and denote by $\ZR(\K)$.
\end{defi}

In some cases, it is useful to assume that the valuations are trivial on some base field. Therefore we have the following  relative version of the Zariski-Riemann space:

\begin{defi}\label{ZR}
Let $\K$  be a field and $\k$ be a subfield of $\K$.
The Zariski-Riemann space of $\K$ modulo $\k$ is the subset of $\ZR(\K)$ of the valuations $\nu$ such that $\nu_{|_{\k}}\equiv 0$.
It is denoted by $\ZR(\K/\k)$.

\end{defi}

\subsection{Topologies on the Zariski-Riemann space}

\begin{defi}
 We define the \emph{Zariski topology} on $\ZR(\K)$ to be the topology generated  by the sets 
$$\O(x):=\{\nu\in\ZR(\K)\mid  \nu(x)\geq 0\}$$
 where $x$ runs over $\K$.
\end{defi}

\begin{thm}\cite[Th\'eor\`eme 7.1]{Va}\label{clos_zariski}
For every $\nu\in\ZR(\K)$ we have
$$\ovl{\{\nu\}}^Z=\{\mu\in\ZR(\K)\mid \mu \text{ is composed with }\nu\}.$$
\end{thm}

\begin{defi}
Let $x\in\K$. We define
$$\U(x):=\{\nu\in\ZR(\K)\mid  \nu(x)> 0\}.$$
The \emph{Inverse topology} on $\ZR(\K)$ is the topology generated by the $\U(x)$ where $x$ runs over the elements of $\K$.
\end{defi}

\begin{defi}
We call \emph{Patch Topology} on $\ZR(\K)$ the topology generated by the sets $\O(x)$ and $\U(x)$ where $x$ runs over $\K$.
\end{defi}

\begin{defi}
Let $\K$  be a field and $\k$ be a subfield of $\K$.
For $x\in\K$ we set 
$$\mathcal{V}(x)=\{\nu\in\ZR(\K/\k)\mid \exists a\in\k,\ \nu(x+a)>0\}.$$
The \emph{Weak Inverse  topology} on $\ZR(\K/\k)$ is the topology generated by the $\mathcal{V}(x)$ where $x$ runs over $\K$. 
\end{defi}

\begin{prop}\label{refi_top}
The Inverse Topology  on $\ZR(\K/\k)$ is finer than the Weak Inverse topology. Both coincide when 
$\k$ is a finite field.
\end{prop}

\begin{proof}
We have $$\V(x)=\{\nu\in\ZR(\K/\k)\mid \exists a\in \k,\ \nu(x+a)>0\}=\bigcup\limits_{a \in \k}\U(x+a).$$
This proves the first claim.
\\
We have that
\begin{equation}\label{eq_topo}\U(x)=\V(x)\cap\bigcap_{a\in\k^*}\V((x+a)^{-1}).\end{equation}
Indeed, if $\nu(x)>0$ and $a\in\k^*$, we have
$$\nu\left(\frac{1}{x+a}-\frac{1}{a}\right)=\nu\left(\frac{-x}{x+a}\right)>0.$$
On the other hand, assume
$$\nu\in \V(x)\cap\bigcap_{a\in\k^*}\V((x+a)^{-1}).$$
Let $a\in\k^*$. Since $\nu \in \V((x+a)^{-1})$, $\nu((x+a)^{-1})\geq 0$. Since $\nu\in\V(x)$, $\nu(x+a)\geq 0$ and $\nu((x+a)^{-1})= 0$. In particular $\nu(x+a)=0$ for every $a\in\k^*$. But, because $\nu\in\V(x)$, we have $\nu(x)>0$. This proves \eqref{eq_topo} and the second claim.
\end{proof}

\begin{defi}
Let $\k\subset \K$ be two fields. Let $\L$ be a field. We denote by $\ZR(\K)^\L$ the subset of valuations of $\ZR(\K)$ whose residue field is $\L$. When $\k\subset \L$, we denote by $\ZR(\K/\k)^\L$ the subset of valuations of $\ZR(\K/\k)$ whose residue field is $\L$.
\end{defi}

\begin{prop}\label{refi_top2}
The Zariski and the Weak Inverse topologies coincide on  $\ZR(\K/\k)^\k$.\\
If $\k$ is a finite field, the Zariski and the Inverse topologies coincide on $\ZR(\K/\k)^\k$.
\end{prop}

\begin{proof}
 Let $x\in\K$.
Let $\nu\in\O(x)\cap\ZR(\K/\k)^\k$, that is,  $\nu(x)\geq 0$.  Since the residue field of $\nu$ is $\k$,  there is $a\in \k$ such that $\nu(x+a)>0$. Therefore $\nu\in\V(x)$.
On the other hand, if $\nu\in \V(x)$, then there is $a\in\k$ such that $\nu(x+a)> 0$. We have that $\nu(a)=0$ or $\infty$, hence $\nu(x)\geq\min\{\nu(x+a),\nu(a)\}\geq 0$, and $\nu\in\O(x)$. This proves that
$$\O(x)\cap \ZR(\K/\k)^\k=\V(x)\cap\ZR(\K/\k)^\k.$$
This proves the first claim. The second claim comes from Proposition \ref{refi_top}.
\end{proof}

\begin{lem}
Let $\k\subset \K$ be two field. Then $\ZR(\K/\k)$ is closed in $\ZR(\K)$ for the Inverse and the Patch topologies.
\end{lem}

\begin{proof}Indeed
$$\ZR(\K/\k)=\bigcap_{x\in\k}(\O(x)\cap\O(x^{-1})).$$
\end{proof}

\subsection{Compacity of the Zariski-Riemann space}

We mention here the following analogue of Theorem \ref{compacity_preorders}. Its proof is completely  similar to the proof of Theorem \ref{compacity_preorders}. Indeed the proof of Theorem \ref{compacity_preorders} is based on the original proof of the following result:

\begin{thm}\cite{SZ}
The spaces $\ZR(\K)$ and $\ZR(\K/\k)$ are  compact for the Zariski, the Inverse and the Patch topologies.
\end{thm}
\subsection{Order on the Zariski-Riemann space}

Therefore we can do exactly as for $\ZR(G)$:

\begin{defi}
Let $\nu_1: \K\lgw G_1$ and $\nu_2:\K\lgw G_2$ be two valuations of $\ZR(\K)$. We say that
$$\nu_1\LEQ\nu_2$$
if there is a valuation $\nu_3$ such that 
$$\nu_2=\nu_1\circ\nu_3.$$
Given two valuations $\nu$ and $\mu\in\ZR(\K)$ we say that $\nu$ and $\mu$ are \emph{comparable} if $\nu\LEQ\mu$ or $\mu\LEQ\nu$. Otherwise we say that they are \emph{incomparable}.
\end{defi}

\begin{lem}\label{comp_val}
Given $\nu$, $\mu\in\ZR(\K)$, the following are equivalent:
\begin{itemize}
\item[i)] $\mu\LEQ \nu$,
\item[ii)] $V_\nu\subset V_\mu$,
\item[iii)] $\m_\mu\subset\m_\nu$.
\end{itemize}
\end{lem}

\begin{proof}
We have $\mu\LEQ \nu$ if and only if $V_\mu$ is a localization of $V_\nu$ at a prime ideal of $V_\nu$. In particular we have $V_\nu\subset V_\mu$.\\
On the other hand, if $V_\nu\subset V_\mu$, the maximal ideal of $V_\mu$, denoted by $\m_\mu$, is a prime ideal of $V_\nu$, and $V_\mu$ is the localization of $V_\nu$ at $\m_\mu$ (see \cite[Proposition 3.3]{Va}). This proves the equivalence of i) and ii).\\
Now if $V_\nu\subset V_\mu$ then $\m_\mu\subset\m_\nu$. On the other hand, if $\m_\mu\subset\m_\nu$, then 
$(\m_\mu\setminus\{0\})^{-1}\subset(\m_\nu\setminus\{0\})^{-1}.$
Thus $V_\nu\subset V_\mu$. This proves the equivalence of ii) and iii).
\end{proof}

\begin{ex}\label{embedding_Zn}
Let $\k$ be a field and $\K=\k(x_1,\ldots, x_n)$ where the $x_i$ are algebraically independent over  $\k$.
Definition \ref{val} shows that there is an injective map 
$$\ZR(\Z^n)\lgw \ZR(\K/\k),$$ whose image is the set of monomial valuations in the coordinates $x_1$, \ldots, $x_n$. It is straightforward to check that this map is continuous for the Zariski, Inverse or Patch topology (when the same topology is considered on both sides), and that this is an increasing map (by Proposition \ref{ca se compose}). Therefore, any choice of generators $x_1$, \ldots, $x_n$ of $\K$ over $\k$ defines such an embedding. 
\end{ex}

\begin{lem}\label{incomparable_val}
Let $\nu$, $\mu\in\ZR(\K)$ be incomparable. Then there exists $f\in \K$ such that $\nu(f)<0$ and $0<\mu(f)$.
\end{lem}

\begin{proof}
By Lemma \ref{comp_val}, $\nu$ and $\mu$ are incomparable if and only if there is $u\in V_\mu\setminus V_\nu$ and $v\in V_\nu\setminus V_\mu$. Therefore we set $f=u/v$ and the claim is proved.
\end{proof}

\begin{rem}
If $\nu_1\LEQ \nu_2$ we have that $G_1$ is the quotient of $G_2$ by a subgroup that is $\nu_2$-isolated.
\end{rem}

\begin{lem}\label{lem_borne_inf}
Let $E\subset\ZR(\K)$ be non empty. The set
$$R_E:=\left\{R \text{ subring of }\K\mid \bigcup\limits_{\nu\in E} V_\nu\subset R\right\}$$
is non empty and contains a minimal element. This minimal element is a valuation ring, and its associated valuation is denoted by $\nu_{\inf E}$. If $E\subset \ZR(\K/\k)$ then $\nu_{\inf E}\in\ZR(\K/\k)$.

\end{lem}

\begin{proof}
The set $E$ is non empty since $\K\in E$. We set $V:=\bigcap\limits_{R\in R_E}R$. Then $V$ is a valuation ring since, for at least one $\nu\in E$, we have $V_\nu\subset V\subset \K$. This proves the lemma.
\end{proof}

\begin{prop}\label{jsl_val}
Let $E\subset\ZR(\K)$ be non empty, and let 
$\nu\in \ZR(\K)$. We have  $$[\forall \mu\in E,\ \nu\LEQ\mu ]\Longleftrightarrow \nu\LEQ\nu_{\inf E}.$$
In particular $\ZR(\K)$ is a join-semilattice, i.e. a partially ordered set in which all subsets have  an infimum.\\
 Moreover, for every $\nu\in\ZR(\K)$, the set
$\{\mu\in\ZR(\K)\mid \mu\LEQ \nu\}$
is totally ordered.\\
The same remains valid if we replace $\ZR(\K)$ by $\ZR(\K/\k)$.
\end{prop}

\begin{proof}
Indeed,  by Lemmas \ref{comp_val} and \ref{lem_borne_inf}, we have
$$ \nu\LEQ\nu_{\inf E} \Longleftrightarrow V_{\nu_{\inf E}}\subset V_\nu\Longleftrightarrow \left[\forall \mu\in E,\   V_\mu\subset V_\nu\right]\Longleftrightarrow [\forall \mu\in E,\ \nu\LEQ\mu ].$$
This proves the first claim. The second claim comes from Lemma \ref{incomparable_val} exactly as in the proof of Theorem \ref{cor_raf_toset}.\end{proof}


\begin{defi}
Let $\K$ be a field and $\nu\in\ZR(\K)$. The rank of $\nu$ is the rank of its value group (that is, the ordinal type of the totally ordered set of its proper isolated subgroups). It is denoted by $\rk(\nu)$.\\
 The degree of $\nu$ is the transcendence degree of $\k_\nu$ over its prime field.\\
When $\k$ is a subfield of $\K$, the (transcendence) degree of $\nu$ is the transcendence degree of $\k\lgw \k_\nu$, and is denoted by $\tdr \nu$.
\end{defi}

\begin{rem}\label{Abhyankar}
Let $\k\lgw \K$ be a field extension of finite transcendental degree. Let $\nu\in \ZR(\K/\k)$ with value group $G$. Then we have
$$\rk(\nu)+\tdr\nu\leq \rr(\nu)+\tdr\nu\leq \tdr(\K)$$
by \cite[Corollary to Theorem 1.20]{Va}. In particular $G$ can be embedded in $\Q^{\tdr(\K)}$. 
\end{rem}

\begin{rem}
Let   $\nu\in\ZR(\K/\k)$. Then $\rk(\nu)>1$ if and only if there exists a non trivial valuation $\nu'\in\ZR(\K/\k)$ such that $\nu'\SEQ\nu$. Therefore the rank one valuations are the minimal valuations $\nu$ such that $\nu_\emptyset\SEQ\nu$.\\
More generally,  $\rk(\nu)$ corresponds to the ordinal type of the maximal chain of valuations between the trivial valuation and $\nu$. Therefore this is the natural analogue of the rank of a preorder.
\end{rem}

\begin{rem}
Let $\nu\in\ZR(\K/\k)$. For any $\ovl\nu\in\ZR(\k_\nu/\k)$, the composition $\nu\circ\ovl\nu$ is well defined. If $\ovl\nu$ is the trivial valuation then $\nu=\nu\circ\ovl\nu$. \\
On the other hand, if $\nu=\nu\circ\ovl\nu$, then $\ovl\nu$ is the trivial valuation. Therefore $\nu\in\ZR(\K/\k)$ is a maximal element if and only if $\ZR(\k_\nu/\k)$ contains only the trivial valuation. And this is the case only if $\k\lgw\k_\nu$ is algebraic.\\
Therefore the maximal elements of $\ZR(\K/\k)$ are the valuations $\nu$ such that $\k\lgw\k_\nu$ is algebraic, that is the valuations of degree 0.\\
More generally, $\tdr\nu$ corresponds to the ordinal type of the maximal chain of valuations between $\nu$ and a valuation $\nu'$ with $\tdr\nu'=0$. Therefore $\tdr\nu$ is the natural analogue of the degree of a preorder.
\end{rem}

\begin{cor}
Let $\k\lgw \K$ be a field extension of finite transcendental degree. Then $\ZR(\K/\k)$ is a rooted graph where the vertices are the valuations on $\ZR(\K/\k)$, the root is the trivial valuation, and for every pair of valuations $(\nu,\mu)$, there is an edge between $\nu$ and $\mu$ if $\nu$ and $\mu$ are comparable and there is no other valuation between them (with respect to the order on $\ZR(\K/\k)$).
\end{cor}

\begin{proof}
This comes directly from the last three remarks and Proposition \ref{jsl_val}, following the same proof as the one of Proposition \ref{tree}.
\end{proof}

\begin{rem}
We can make the similar reasoning for $\ZR(\K)$. A valuation $\nu\in\ZR(\K)$ has no refinement if and only if $\ZR(\k_\nu)$ contains only the trivial valuation. But any characteristic zero field contains non trivial valuations (any $p$-adic valuation on $\Q$, and any extension of it on a characteristic zero field). For $p>0$,  $\ZR(\F_p)$ contains only the trivial valuation, and this remains true for $\ZR(\K)$ when $\F_p\lgw \K$ is algebraic. Therefore, the maximal elements $\nu$ of $\ZR(\K)$ are the valuations for which $\k_\nu$ is an algebraic extension of $\F_p$.
\end{rem}

Now we can prove the analogue of Theorem \ref{clos_zariski} for the Inverse topology and the Weak Inverse topology:
\begin{thm}
Let $\K$ be a field and $\k$ a subfield of $\K$. We have:
$$\forall\nu\in\ZR(\K/\k),\ \ 
\ovl{\{\nu\}}^{I}=\ovl{\{\nu\}}^{WI}=\{\mu\in\ZR(\K/\k)\mid \mu \LEQ\nu\},$$
$$\forall\nu\in\ZR(\K),\ \ 
\ovl{\{\nu\}}^{I}=\{\mu\in\ZR(\K)\mid \mu \LEQ\nu\}.$$
\end{thm}

\begin{proof}
Let $\mu\LEQ\nu$, that is $V_\nu\subset V_\mu$.  Let $x\in \K$ such that $\mu\in\U(x)$. Then there is $a\in\K$ such that $\mu(x+a)>0$. Therefore $\nu(x+a)> 0$ since $\m_\mu\subset \m_\nu$, and $\nu\in\U(x)$. Therefore $\{\mu\in\ZR(\K/\k)\mid \mu \LEQ\nu\}\subset \ovl{\{\nu\}}^{WI}$.\\
Now, if $\mu$ and $\nu$ are incomparable, there is $x\in\K$ such that $\mu(x)>0$ and $\nu(x)<0$ by Lemma \ref{incomparable_val}. Therefore $\mu\in\U(x)$ and $\nu\notin \U(x)$. Hence $\mu\notin \ovl{\{\nu\}}^{WI}$.\\
Finally, let $\nu\LEQ\mu$, that is $\m_\nu\subset\m_\mu$. Let $x\in\m_\mu$, i.e. $\mu(x)>0$. Then $\mu\in\U(x)$. If $\mu\in \ovl{\{\nu\}}^{WI}$, then $\nu\in\U(x)$, and there is $a\in\k$ such that $\nu(x+a)>0$. Therefore, by hypothesis, $\mu(x+a)>0$ and necessarily $a=0$. This shows that $\m_\mu\subset \m_\nu$, hence $\m_\mu=\m_\nu$ and $\mu=\nu$. This proves the result for the Weak Inverse topology.\\
For the Inverse topology, the proof is similar.
\end{proof}

\begin{rem}
Let $\k\lgw\K$ be a field extension. The analogue of the action of $\Aut(G)$ over $\ZR_*(G)$, is the left action of $\Gal(\K/\k)$ over $\ZR(\K/\k)$ defined as follows:
$$\forall \s\in\Gal(\K/\k), \forall \nu\in\ZR(\K,\k),\forall x\in \K,\ (\s\cdot\nu)(x):=\nu(\s^{-1}(x)).$$
For instance, if $\K=\k(x_1,\ldots, x_n)$ where the $x_i$ are algebraically independent over $\k$, $\Gal(\K/\k)$ is the Cremona group $Cr_n(\k)$ of $\mathbb P^n(\k)$. This group contains the subgroup of monomial bijections of the form
$$(x_1,\ldots, x_n)\lgm (x_1^{a_{11}}x_2^{a_{12}}\ldots x_n^{a_{1n}}, \ldots, x_1^{a_{n1}}x_2^{a_{n2}}\ldots x_n^{a_{nn}})$$
where the matrix $(a_{ij})\in\Mat_n(\Z)$  is invertible in $\Mat_n(\Z)$. Therefore $\Aut(\Z^n)\subset Cr_n(\k)$. Moreover the action of $\Aut(\Z^n)$ on $\ZR(\Z^n)$ is induced by the action of $\Aut(\Z^n)$ on $\ZR(\K/\k)$ via the embedding introduced in Example \ref{embedding_Zn}.
\end{rem}


\subsection{Metric on the Zariski space in the countable case}

\begin{defi}\label{filtration}
Let $\K$ be a countable field, and let $\{F_n\}_{n\in\N}$ be a filtration of $\K$ by finite sets. That is, the $F_n\subset \K$ are finite, $F_n\subset F_{n+1}$ for every $n$, and $\bigcup\limits_n F_n=\K$. Moreover we assume that, for all $x\in F_n$, $x\neq 0$, we have $x^{-1}\in F_n$.\\
For $x\in \K$ we set $\haut(x):=\min\{n\in\N\mid x\in F_n\}$.
\end{defi}

\begin{defi} 
Let $\K$ be a countable field and $\{F_n\}$ be a filtration of $\K$ as in Definition \ref{filtration}.\\
 For $\nu$, $\mu\in \ZR(\K)$, $\nu\neq\mu$, we set
$$d(\nu,\mu)=\frac{1}{n}$$
if for every $x\in\K$ with $\haut(x)< n$, we have
$$\nu(x)>0\Longrightarrow \mu(x)>0$$
$$\nu(x)=0\Longrightarrow \mu(x)=0,$$
and there is $x\in\K$, with $\haut(x)=n$, such that one of these implications is not satisfied.\\
If $\nu=\mu$ we set
$$d(\nu,\mu)=0.$$
\end{defi}


\begin{thm}\label{patch_metric}
Let  $\K$ be a countable field. We have
\begin{itemize}
\item[i)] The function $d$ is an ultrametric on $\ZR(\K)$.
\item[ii)] The topology induced by $d$ coincides with the Patch topology on $\ZR(\K)$. In particular it does not depend on the choice of the filtration $\{F_n\}_{n\in\N}$.
\end{itemize}
\end{thm}

\begin{proof}
Clearly $d$ is non negative, reflexive and symmetric. The ultrametric inequality is straightforward to check.
Therefore we only need to prove ii).\\
\\
Now let $n\in\N^*$ and $\nu$, $\mu\in \ZR(\K)$. \\
For all $x\in \K$, let $\nu\in\U(x)$. Then
$$B\left(\nu, \frac{1}{\haut(x)}\right)\subset \U(x).$$
Indeed, if $\mu\in B\left(\nu, \frac{1}{\haut(x)}\right)$, we have
$$\nu(x)>0\Longrightarrow \mu(x)>0.$$
Hence the $\U(x)$ are open for the topology induced by $d$, and the topology induced by $d$ is finer than the  I-topology on $X$.\\
Now let $x\in\K$, $\nu\in\O(x)$, and $\mu\in B\left(\nu, \frac{1}{\haut(x)}\right)$. Then we have
$$\nu(x)\geq 0\Longrightarrow  \mu(x)\geq 0.$$
Therefore
$$B\left(\nu, \frac{1}{\haut(x)}\right)\subset \O(x),$$
and the topology induced by $d$ is finer than the  Z-topology on $X$. Hence, the topology induced by $d$ is finer than the Patch topology.\\
\\
On the other hand, we have that
 $\mu\in B(\nu,\frac{1}{n})$ if and only if, for every $x\in\K$ with $\haut(x)\leq n$,
$$\nu(x)>0\Longrightarrow \mu(x)>0,$$
$$\nu(x)=0\Longrightarrow \mu(x)=0.$$
 Therefore we have
 $$B\left(\nu,\frac{1}{n}\right)=\bigcap_{x,\haut(x)\leq n, \nu(x)>0}\U(x)\cap  \bigcap_{x,\haut(x)\leq n, \nu(x)=0}\left(\O(x)\cap\O(x^{-1})\right).$$
And this ball is open in the  Patch topology because this intersection is finite. Therefore both topologies coincide.
\end{proof}

\begin{cor} Let $\k\lgw\K$ be a field extension where $\k$ is a finite field and $\K$ is countable. Then
 the Zariski topology on $\ZR(\K/\k)^\k$ is a metric topology.
\end{cor}

\begin{proof}
This comes from Theorem \ref{patch_metric} and Proposition \ref{refi_top2}.
\end{proof}

\subsection{Cantor sets}
We have the following lemma:
\begin{lem}\label{tot_disc_val}
Let $E\subset \ZR(\K/\k)$. Then $E$ is totally disconnected for the Patch topology.
\end{lem}

\begin{proof}
Let $\nu$, $\mu\in\ZR(\K/\k)$, $\nu\neq \mu$. Therefore $V_\nu\neq V_\mu$; for instance $V_\mu\nsubseteq V_\nu$. Thus  there is $x\in \K$ such that $\nu(x)<0$ and $\mu(x)\geq 0$. Thus, $\nu\in \U(x^{-1})$ and $\mu\in\O(x)$. But
$$\U(x^{-1})\cup \O(x)=\ZR(\K/\k) \text{ and }\U(x^{-1})\cap\O(x)=\emptyset.$$
This proves the claim.
\end{proof}
Therefore, when $\K$ is a countable field, $\ZR(\K/\k)$ is a metric compact totally disconnected space for the Patch topology. A natural question is to investigate when this is a Cantor space, or when a closed subset $E$ of $\ZR(\K/\k)$ is a Cantor space. This happens if and only if $\ZR(\K/\k)$ (or $E$) is a perfect space.
\begin{ex}
When $x$ is a single indeterminate and $\k$ is algebraically closed, $\ZR(\k(x),\k)$ is not a perfect space. Indeed, for every $\nu\in\ZR(\k(x),\k)$, $\nu$ being non trivial, there is a unique $y\in \K$ such that $\K=\k(y)$ and $\nu(y)>0$. Such a $y$ can be chosen as $x^{-1}$ or $x+a$  for some $a\in\k$. Moreover, for such a $y$, there is a unique valuation $\nu\in\ZR(\K/\k)$ such that $\nu(y)>0$, since $\k$ is algebraically closed. We denote by $\nu_{a}$  the unique valuation such that $\nu(x+a)>0$, and by $\nu_-$ the unique valuation such that $\nu_-(x)<0$. Therefore we have 
$$\ZR(\k(x)/\k)=\{\nu_0\}\cup\{\nu_-\}\bigcup_{a\in\k}\{\nu_{a}\}.$$
Moreover, $\{\nu_-\}=\U(x^{-1})$
and, for every $a\in\k$, 
$\{\nu_{a}\}=\U(x+a)$
are open sets.
\end{ex}


\begin{ex}
Let $\k$ be a finite or countable field, and $\K=\k(x,y)$ where $x$ and $y$ are algebraically independent over $\k$. Let $\nu$ be the monomial valuation defined by
$$\nu(x)=1 \text{ and } \nu(y)=1.$$
We have $\tdr\nu=1$.
Then we claim that $\nu$ is the limit of valuations  of transcendence degree 0. In particular the inequality about the degree in Corollary \ref{conv3} does not hold for valuations. To show this, we consider two cases (depending on whether  $\k$ is finite or not):\\
$\bullet$ If $\k$ is countable, we consider a filtration of $\k$ by finite sets $\k_n$, and we set
$$F_n:=\{P/Q\mid P,Q\in \k_n[x,y], \deg(P), \deg(Q)\leq n\}.$$
Since the $\k_n$ are finite, we may choose, for every integer $n$, $a_n\in\k$ such that $x+a_ny$ does not divide any nonzero homogeneous form of any polynomial $P\in\k_n[x,y]$ of degree $\leq n$. We denote by $\nu_n$ the monomial valuation defined by
$$\nu_n(y)=1 \text{ and } \nu_n(x+a_ny)=\sqrt{2}.$$ 
Then, for $P\in\k_n[x,y]$ of degree $\leq n$, we write
$$P=P_k(x,y)+P_{k+1}(x,y)+\cdots$$
where $P_j$ is a homogeneous polynomial of degree $j$, and $P_k\neq 0$. Then $\nu(P)=k$.\\
Now
$$P_j(x,y)=c_{j}y^j+P_{j,1}(x,y)(x+a_ny)$$
with $c_j\in \k$.
Since  $x+a_ny$ does not divide  any nonzero homogeneous form of any polynomial in $\k_n[x,y]$ of degree $\leq n$, we have that $c_j\neq 0$ as soon as $P_j\neq 0$. Since $P_{j,1}(x,y)$ is a homogeneous polynomial of degree $j-1$, we have $\nu_n(P_j)=j$ when $P_j\neq 0$, and $\nu_n(P)=k$. Therefore 
$$\nu_n(R)=\nu(R), \ \ \ \forall R\in F_n.$$
This shows that the sequence $(\nu_n)$ converges to $\nu$ for the Patch topology. We remark that the $\nu_n$ are rational valuations and $\nu$ is not (the transcendence degree of $\nu$ is 1). This shows that $\ZR(\k(x,y)/\k)^{\k}$ is not closed when $\k$ is infinite. Even more, this shows that the set of valuations of transcendence degree equal to  0 is not closed for the Patch topology.\\
$\bullet$  If $\k$ is finite, we consider a filtration of $\k$ by finite sets $\k_n$ as before, and we set
$$F_n:=\{P/Q\mid P,Q\in \k_n[x,y], \deg(P), \deg(Q)\leq n\}.$$
For every integer $n$, we consider an irreducible polynomial $P_n(T)\in\k[T]$ of degree $>n$. The polynomial $p_n(x,y):=y^{\deg(P_n)}P_n(x/y)$ is an irreducible  homogenous polynomial of degree $\deg(P_n)>n$. For every $f\in\k[x,y]$, we consider the $p_n$-expansion of $f$:
$$f=\sum_{i=0}^kf_ip_n^i$$
where $\deg(f_i)<\deg(p_n)$ for every $i$. Then we define the valuation $\nu_n$ by
$$\nu_n(f):=\min\{\nu(f_i)+\d_n i\}$$
where $\d_n>\deg(P_n)$ and $\d_n\in\R\setminus\Q$. Then $\nu_n\neq \nu$ and the sequence $(\nu_n)_n$ converges to $\nu$ for the Patch topology.
Moreover, in this case, $\k_{\nu_n}\simeq \k[T]/(P_n(T))$ is a non trivial algebraic extension of $\k$. Thus we have
$$\tdr\nu_n=0 \text{ and } \rr(\nu_n)=\dim_{\Q}(\Q+\Q\d_n)= 2.$$
But $\tdr\nu=1$, therefore, the set of valuations of transcendence degree equal to  0 is not closed for the Patch topology.\\
\\
Let us remark that this example can be easily extended to $\k(x_1,\ldots, x_n)$ where the $x_i$ are algebraically independent over $\k$, by considering the monomial valuation $\nu$ defined by
$$\nu(x_1)=\nu(x_2)=1$$
and choosing the $\nu(x_i)$, $i\geq 3$, such that
$1,\nu(x_3),\ldots, \nu(x_n)$
are $\Q$-linearly independent.
\end{ex}

\begin{rem}
In the previous example we have $\rk(\nu_n)=2$ while $\rk(\nu)=1$. Therefore the inequality about the rank in Corollary \ref{conv3} does not hold for valuations.
\end{rem}

We have the following lemma:

\begin{lem}\label{rat_closed}
Let $\k$ be a finite field and $\K$ be any field extension of $\k$. Then $\ZR(\K/\k)^\k$ is a compact subset of $\ZR(\K,\k)$ for the Zariski Topology. 
\end{lem}

\begin{proof}
We remark that
$$\nu\notin\ZR(\K/\k)^{\k}\Longleftrightarrow \exists y\in\K,\ \forall a\in \k,\ \  \nu(y+a)=0.$$
Therefore
$$\ZR(\K/\k)^\k=\left(\bigcup_{y\in \K}\bigcap_{a\in\k}(\O(y+a)\cap\O((y+a)^{-1}))\right)^c$$
is closed if $\k$ is finite for the Zariski Topology by Proposition \ref{refi_top}. Therefore it is compact.
\end{proof}

Therefore we can formulate the following conjecture.\\

\noindent\textbf{Conjecture A.} Let $\k$ be a field, and let $\K$ be a countable field extension of $\k$ of transcendence degree at least 2.  Then $\ZR(\K/\k)^\k$ is a perfect set for the Patch Topology. Therefore, when $\k$ is finite, it is a Cantor set for the Patch and the Zariski Topologies.\\
\\
We give a proof of this conjecture in the following case:

\begin{thm}\label{cantorset1}
Let $n\geq 2$ and $\k$ be a countable field. 
 Then the  set $\ZR(\k(x_1,\dots,x_n)/\k)^{\k}$ is a totally disconnected perfect metric set for the Patch topology.
Moreover, if $\k$ is finite, it is a Cantor set for the Patch and the Zariski topologies.

\end{thm}

\begin{proof}
By Lemma \ref{tot_disc_val}, $\ZR(\k(x_1,\dots,x_n)/\k)^{\k}$ is totally disconnected. Since $\k$ is countable, the patch topology is a metric topology by Theorem \ref{patch_metric}. Therefore, we only need to prove that $\ZR(\k(x_1,\dots,x_n)/\k)^{\k}$ is a perfect  space.\\
\\
 Now assume that $\ZR(\k(x_1,\ldots,x_n)/\k)^\k$ is not perfect. Thus,  there exist $a_1,\dots,a_s,b_1,\dots,b_m \in \k(x_1,\dots,x_n)$ such that the set 
$$E:=\bigcap\limits_{i=1}^s\O(a_i)\cap \bigcap\limits_{j=1}^m\U(b_j)$$
is finite and no empty. Even if it means to add some points $a_i$ or $b_j$, we may assume that $E$ has exactly one element, that we denote by $\nu$.\\
 Since $\nu$ is rational, for all $i$, there exists $\lambda_i \in \k$ such that $\nu(\lambda_i +a_i)>0$. Therefore 
 $$\nu\in \bigcap\limits_{i=1}^s\U(a_i+\la_i)\cap \bigcap\limits_{j=1}^m\U(b_j)\subset A=\bigcap\limits_{i=1}^s\O(a_i)\cap \bigcap\limits_{j=1}^m\U(b_j).$$
Hence, we may assume that
$$E=\{\nu\}= \bigcap\limits_{j=1}^{m}\U(b_j).$$
Let $T$ be a key polynomial associated to $\nu$ with respect to the variable $x_n$. For every polynomial $P\in \k(x_1,\dots,x_{n-1})[x_n]$, we consider the $T$-expansion of $P$:
$$P=\sum \limits_{l=0}^dp_lT^l$$
with $\deg_{x_n}(p_l)<\deg_{x_n}(T)$ for all $l$. Let $G$ be an ordered group strictly containing $\G_\nu\otimes_\Z\Q$, and let $\d\in G\setminus \G_\nu\otimes_\Z\Q$ be such that $\d>\nu(T)$. We set
$\nu_\delta(P)=\min\limits_{0\leq l \leq d}\{\nu(p_l)+l\delta\}$.\\
Since $T$ is a key polynomial,  by \cite{Va1}[Lem 1.1], for such a  $P\in\k(x_1,\dots,x_{n-1})[x_n]$, we have
$$\nu(P)=\min\limits_{0\leq l \leq d}\{\nu(p_l)+l\nu(T)\}.$$
Let $r$ be the least integer such that $\nu(P)=\nu(p_r)+r\nu(T)$. Then, for $\d-\nu(T)>0$ small enough, we still have $\nu_\d(P)=\nu(p_r)+r\d$.\\
 Let $Q\in \k(x_1,\dots,x_{n-1})[x_n]$, whose $T$-expansion is $Q=\sum_{l=0}^eq_l T^l$, and let $s$  be the least integer such that $\nu(Q)=\nu(q_s)+s\nu(T)$. Assume that
$\nu(P/Q)>0$. Then for $\d-\nu(T)>0$ small enough, we have
$$\nu_\d(P/Q)=\nu_\d(P)-\nu_\d(Q)=\nu(p_r)-\nu(q_s)+(r-s)\d>0.$$
If $r-s>0$, we have $\nu_\d(P/Q)>\nu(P/Q)>0$. If $r-s<0$, we have $\nu(P/Q)>0$, hence $\frac{\nu(p_r)-\nu(q_s)}{s-r}>\nu(T)$. Therefore, for $\d<\frac{\nu(p_r)-\nu(q_s)}{s-r}$ we have $\nu_\d(P/Q)>0$.\\
 We write $b_j=\frac{P_j}{Q_j}$ for every $j$, where $P_j$, $Q_j\in \k(x_1,\ldots, x_{n-1})[x_n]$. Then, for $\d-\nu(T)>0$ small enough, we have $\nu_\d(P_j/Q_j)>0$. Moreover, since $\d\notin \G_\nu\otimes_\Z\Q$, the valuation $\nu_\d$ is a rational valuation by \cite[Th\'eor\`eme 1.12, Proposition 1.13]{Va1} (see also \cite[Theorem 12.1]{McL}), hence $\nu_\d\in E$. Finally, since $\d>\nu(T)$, $\nu_\d\neq\nu$. This contradicts the fact that $E$ contains only one element in $\ZR(\k(x_1,\ldots, x_n)/\k)^\k$. Thus $\ZR(\k(x_1,\ldots, x_n)/\k)^\k$ is a perfect space.\\
 Hence it is a Cantor set, when $\k$ is finite by Lemma \ref{rat_closed} (for both topologies by Proposition \ref{refi_top2}).
 \end{proof}

For $\ZR(\K/\k)$ we have the following conjecture.\\

\noindent\textbf{Conjecture B.} Assume that $\k$ is a finite or countable field and $\K$ a countable field extension of $\k$ of transcendence degree at least 2. Then $\ZR(\K/\k)$ is a Cantor set for the Patch topology.\\
\\
We prove here an important case of this conjecture:


\begin{thm}\label{cantorset2}
Assume that $\k$ is a finite or countable field and $\K$ a finitely generated field extension of $\k$ of transcendence degree at least 2. Then $\ZR(\K/\k)$ is a Cantor set for the Patch topology.
\end{thm}

\begin{proof}
As for Theorem \ref{cantorset1}, we only have to show that $\ZR(\K/\k)$ is a perfect space. Assume, aiming for contradiction, that this space is not perfect. Then there exist $a_1,\dots,a_s,a'_1,\dots,a'_m\in \K$ such that
$$E:=\bigcap\limits_{i=1}^s\O(a_i)\cap \bigcap\limits_{j=1}^m\U(a'_j)$$
is non empty and contains a unique element $\nu$. Let $x_1$, \ldots, $x_d$ be elements of $\K$ such that $\K=\k(x_1,\ldots, x_d)$. By replacing $x_k$ by $x_k^{-1}$, we may assume that $\nu(x_k)\geq 0$ for every $k$. We denote by $A$ the $\k$-algebra generated by the $x_k$, the $a_i$ and the $a'_j$.
 Then $A$ is an integral domain whose field of fractions is $\K$. We have that, for every $a\in A$, $\nu(a)\geq 0$. We set 
 $$I:=\{a\in A\mid \nu(a)>0\}.$$
 This is a prime ideal of $A$ containing the $a'_j$. We denote by $B$ the normalization of $A$. Then $B$ is Noetherian since $A$ is finitely generated over $\k$ (see \cite[Theorem 4.6.3]{HS} for instance). Moreover the ideal $\p:=\{b\in B\mid \nu(b)\geq 0\}$ is a prime ideal of $B$ lying over $I$.  \\ 
 We have that $\dim(B_\p)=\dim(A)=\tdr(\K)=n\geq 2$. Then $B_\p$ is not a valuation ring since $B_\p$ is Noetherian, therefore there exists $y\in\K$ such that $y\notin B_\p$ and $y^{-1}\notin B_\p$. \\
 We claim that $yB_\p[y]+\p B_\p[y]\neq B_\p[y]$. Indeed, if
 $$1=y(b_0+b_1y+\cdots+b_my^m)+p_0+p_1y+\cdots+p_my^m$$
 for some $b_k\in B_\p$ and $p_l\in \p$, we would have 
 $$(1-p_0)y^{-m-1}-(b_0+p_1)y^{-m}+\cdots-(b_{m-1}+p_m)y^{-1}-b_m=0$$
 and $y^{-1}$ would be integral over $B_\p$ since $1-i_0\notin \p$ is invertible in $B_\p$. But this is not possible since $B_\p$ is  integrally closed and $y^{-1}\notin B_\p$.\\
 In the same way, $y^{-1}B_\p[y^{-1}]+\p B_\p[y^{-1}]\neq B_\p[y^{-1}]$.\\
 Now let $\q_1$ (resp. $\q_2$) be a prime ideal of $B_\p[y]$ (resp. $B_\p[y^{-1}]$) containing $yB_\p[y]+\p B_\p[y]$ (resp. $y^{-1}B_\p[y^{-1}]+\p B_\p[y^{-1}]$). Then there exists a valuation ring $V_1$ in $\K$, whose maximal ideal $\m_{V_1}$ satisfies $\m_{V_1}\cap B_\p[y]=\q_1$ (see for example \cite[Theorem 6.3.2]{HS}). Therefore the associated valuation $\nu_1$ satisfies 
 $\forall q\in\q, \nu_1(q)>0$. Therefore $\nu_1(a_i)\geq 0$ for every $i$, and $\nu_1(a'_j)>0$ for every $j$.
 In the same way, there exists a valuation ring $V_2$ in $\K$, whose maximal ideal $\m_{V_2}$ satisfies $\m_{V_2}\cap B_\p[y^{-1}]=\q_2$, and its associated valuation $\nu_2$ satisfies  $\nu_1(a_i)\geq 0$ for every $i$, and $\nu_1(a'_j)>0$ for every $j$.
 But $\nu_1\neq \nu_2$ because $\nu_1(y)>0$ and $\nu_2(y^{-1})>0$. This contradicts the fact that $E$ is a singleton. Therefore $\ZR(\K/\k)$ is a perfect set. 
 \end{proof}

\FloatBarrier

\end{document}